\documentclass[11pt]{amsart}

\usepackage{a4wide}
\usepackage{amssymb}
\usepackage{mathrsfs}
\usepackage{enumerate}
\usepackage{esint}
\usepackage{titletoc}
\usepackage{mathtools}
\usepackage[colorlinks=true, urlcolor=blue, linkcolor=blue, citecolor=green]{hyperref}
\usepackage[nameinlink]{cleveref}

\usepackage{mathtools}
\usepackage{enumitem}
\usepackage[normalem]{ulem}
\usepackage{bm}

\usepackage{kotex}

\theoremstyle{plain}
\newtheorem{theorem}{Theorem}[section]
\newtheorem{lemma}[theorem]{Lemma}
\newtheorem{corollary}[theorem]{Corollary}

\theoremstyle{definition}
\newtheorem{definition}[theorem]{Definition}

\numberwithin{equation}{section}
\DeclareMathOperator*{\osc}{osc}

\usepackage{nameref}
\makeatletter
\let\orgdescriptionlabel\descriptionlabel
\renewcommand*{\descriptionlabel}[1]{%
	\let\orglabel\label
	\let\label\@gobble
	\phantomsection
	\edef\@currentlabel{#1}%
	\let\label\orglabel
	\orgdescriptionlabel{#1}%
}
\makeatother
\numberwithin{equation}{section}

\makeatletter
\@namedef{subjclassname@2020}{\textup{2020} Mathematics Subject Classification}
\makeatother

\author{Ki-Ahm Lee}
\address{Department of Mathematical Sciences and Research Institute of Mathematics,
	Seoul National University, Seoul 08826, Republic of Korea.}
\email{kiahm@snu.ac.kr}

\author{Se-Chan Lee}
\address{Research Institute of Mathematics,
	Seoul National University, Seoul 08826, Republic of Korea.}
\email{dltpcks1@snu.ac.kr}

\author{Hyungsung Yun}
\address{Department of Mathematical Sciences,
	Seoul National University, Seoul 08826, Republic of Korea.}
\email{euler@snu.ac.kr}

\thanks{Ki-Ahm Lee is supported by the National Research Foundation of Korea (NRF) grant funded by the Korea government (MSIP): NRF-2020R1A2C1A01006256. Se-Chan Lee is supported by Basic Science Research Program through the National Research Foundation of Korea (NRF) funded by the Ministry of Education (2022R1A6A3A01086546).}

\begin{document}
\title[Degenerate/singular fully nonlinear parabolic equations]{$C^{1, \alpha}$-regularity for solutions of degenerate/singular fully nonlinear parabolic equations}
\maketitle

\begin{abstract}
	We establish the interior $C^{1,\alpha}$-estimate for viscosity solutions of degenerate/singular fully nonlinear parabolic equations
$$u_t=|Du|^{\gamma}F(D^2u) + f.$$
For this purpose, we prove the well-posedness of the regularized Dirichlet problem
\begin{equation*} 
	\left\{ \begin{aligned}
		u_t&=(1+|Du|^2)^{\gamma/2}F(D^2u) &&\text{in $Q_1$}\\
		u&=\varphi &&\text{on $\partial_p Q_1$}.
	\end{aligned}\right.
\end{equation*}
Our approach utilizes the Bernstein method with approximations in view of difference quotient.
\end{abstract}
%
%
\section{Introduction}\label{sec:introduction}
This paper is devoted to the study of the $C^{1, \alpha}$-regularity for viscosity solutions of the following degenerate/singular fully nonlinear parabolic equations
\begin{align} \label{eq:model_f}
	u_t=|Du|^{\gamma}F(D^2u) + f,
\end{align}
where $\gamma>-1$, and $f$ is bounded and continuous. Here the operator $F$ is uniformly elliptic with certain structural conditions (the hypotheses on $F$ will be precisely stated in \Cref{sec:preliminaries}). The fully nonlinear parabolic equation \eqref{eq:model_f} is motivated by the Hamilton-Jacobi-Bellman equations of the time-dependent, two-player stochastic differential games. In other words, the viscosity solution of \eqref{eq:model_f} can be realized as the value function of an associated stochastic control problem; see e.g. \cite{CSTV07, FS93, FS89, KS10}.

Thanks to the reduction scheme presented in \cite{Att20, AR20}, it suffices to concentrate on the following homogeneous equations
\begin{align} \label{eq:model}
	u_t=|Du|^{\gamma}F(D^2u)
\end{align}
instead of nonhomogeneous one \eqref{eq:model_f}. Therefore, our main theorem is concerned with $C^{1, \alpha}$-regularity of viscosity solutions $u$ of \eqref{eq:model} with uniform estimates provided that $F$ is convex.
\begin{theorem}\label{thm:C1alpha}
	Assume that  $\gamma>-1$, $F$ satisfies \ref{F1}, \ref{F2}, and \ref{F3}, and $f$ is bounded and continuous in $Q_1$. Let $u$ be a viscosity solution of \eqref{eq:model_f} in $Q_1$. Then there exists a constant $ \alpha \in (0,1)$ depending only on $n$, $\lambda$, $\Lambda$, $\gamma$, $\|u\|_{L^{\infty}(Q_1)}$, and $\|F\|_{C^{1,1}(\mathcal{S}^n)}$ such that $u \in C^{1, \alpha}(\overline{Q_{1/2}})$ with an estimate
	\begin{align*}
		\|Du\|_{C^{\alpha} (\overline{Q_{1/2}})} \leq C,
	\end{align*}
where $C>0$ is a constant depending only on $n$, $\lambda$, $\Lambda$, $\gamma$, $\|u\|_{L^{\infty}(Q_1)}$, and $\|f\|_{L^{\infty}(Q_1)}$. Moreover, it holds that 
\begin{equation*} 
	|u(x,t)-u(x,s)| \leq C |t-s|^{\frac{1+\alpha}{2-\alpha \gamma}}  \quad \text{for all } (x,t),(y,s) \in Q_{1/2}.
\end{equation*}
\end{theorem}
In order to prove \Cref{thm:C1alpha}, we investigate the solvability of the Dirichlet problem which is generated by the regularization of \eqref{eq:model}:
\begin{equation} \label{eq:mean}
	\left\{ \begin{aligned}
		u_t&=(1+|Du|^2)^{\gamma/2}F(D^2u) &&\text{in $Q_1$}\\
		u&=\varphi &&\text{on $\partial_p Q_1$}.
	\end{aligned}\right.
\end{equation}
\begin{theorem}\label{thm:solv}
	Suppose that $\gamma>-2$, $F$ satisfies \ref{F1}, \ref{F2}, and \ref{F3}. Let $\varphi \in C^{2, \beta}(\overline{Q_1})$ for some $\beta \in (1/2, 1)$, and satisfy the compatibility condition:
	\begin{align*}
		\varphi_t=(1+|D\varphi|^2)^{\gamma/2}F(D^2 \varphi) \quad \text{on $\partial_c Q_1$},
	\end{align*}	
	where $\partial_c Q_1 \coloneqq \partial B_1  \times \{ -1 \}$. Then the Dirichlet problem \eqref{eq:mean} is uniquely solvable in $C^{2, \beta}(\overline{Q_1})$.
\end{theorem} 
It is noteworthy that \eqref{eq:mean} can be understood as one of the fully nonlinear generalizations of the mean curvature equation
\begin{equation} \label{eq:mcf}
	u_t = \textnormal{div} \, \frac{Du}{(1+|Du|^2)^{1/2}} = (1+|Du|^2)^{-1/2} a_{ij}(Du) D_{ij} u,
\end{equation}
where $a_{ij}(p) \coloneqq \delta_{ij} - p_i p_j/(1+|p|^2)$. Here we remark that
\begin{equation*}
	\frac{|\xi|^2}{1+K^2} \leq \frac{|\xi|^2}{1+|p|^2} \leq a_{ij}(p) \xi_i \xi_j \leq |\xi|^2 \quad \text{for $\xi \in \mathbb{R}^n$ and $|p| \leq K$.}
\end{equation*}
In other words, the ellipticity $\Lambda/\lambda$ of mean curvature-type equations is finite provided that $Du$ is bounded, but it is heavily influenced by $Du$.

Let us summarize the results of the preceding literature which deals with degenerate/singular equations in non-divergence form. We first collect the solvability results of Dirichlet problems in different settings. If we replace the first equation in \eqref{eq:mean} by quasilinear equations such as
\begin{equation}\label{eq:IJS19_reg}
	u_t =(\varepsilon^2+|Du|^2)^{\gamma/2}\left(\delta_{ij} + \frac{(p-2)D_i u D_j u }{\varepsilon^2+|Du|^2} \right)D_{ij}u 
\end{equation}
or $u_t = a_{ij}(x,t,u,Du)D_{ij} u + b(x,t,u,Du)$, then the existence, uniqueness, and smoothness of the solution were illustrated in the comprehensive books \cite{LSU68, Lie96} and references therein. Moreover, for fully nonlinear parabolic operators $F=F(x, t, r, p, M)$ enjoying several structural conditions, the well-posedness of corresponding Dirichlet problems was formulated in \cite{CKLS99, CKS00, dST17, Kry18}. Nonetheless, in the aforementioned papers, the uniform Lipschitz continuity with respect to the gradient variable $p$ is the main assumption on $F$, which is not satisfied by our operator $F(x, t, r, p, M)=(1+|p|^2)^{\gamma/2}F(M)$. Finally, the existence, uniqueness, and global H\"older regularity results for viscosity solutions of the Dirichlet problem
\begin{equation*}
	\left\{ \begin{aligned}
		u_t&=|Du|^{\gamma}F(D^2u) &&\text{in $Q_1$}\\
		u&=\varphi &&\text{on $\partial_p Q_1$}.
	\end{aligned}\right.
\end{equation*}
were demonstrated in \cite{BM17, Dem11} under appropriate structural conditions on $F$, $f$, and $\varphi$. However, to the best of our knowledge, the higher regularity than H\"older continuity of viscosity solutions $u$ of \eqref{eq:model_f} is unknown in literature.

Let us move on to the interior regularity results. For the quasilinear variant instead of $F$ in \eqref{eq:model_f}, the $C^{1, \alpha}$-regularity for a viscosity solution $u$ of 
\begin{equation} \label{eq:IJS19}
	u_t=|Du|^{\gamma}\Delta_p^Nu\coloneqq|Du|^{\gamma}  (\delta_{ij} + (p-2) |Du|^{-2} D_i u D_j u   ) D_{ij}u
\end{equation}
was established in \cite{IJS19, JS17} when $p>1$ and $\gamma>-1$. They suggested two alternatives to describe the oscillation of $Du$ with uniform estimates for approximated solutions; we will present a more precise explanation of their analysis, and compare it with ours later. In \cite{FZ21, FZ23}, they provided similar consequences for an extended class of quasilinear parabolic equations. As mentioned before, \cite{Att20, AR20} verified that the interior $C^{1, \alpha}$-regularity result for homogeneous equations can be transferred to the one for nonhomogeneous equations, in both fully nonlinear and quasilinear settings.

On the other hand, the elliptic analogue of \eqref{eq:model_f} has been relatively widely studied in the last decade. To be precise, the local $C^{1, \alpha}$-regularity result for 
\begin{equation} \label{eq:IS13}
	|Du|^{\gamma}F(D^2u)=f
\end{equation}
was developed in \cite{IS13}. Later, the optimality of the exponent $\alpha$ \cite{ART15} and the global regularity result \cite{BD14} were investigated for a similar class of elliptic operators. There are many recent papers on fully nonlinear elliptic equations with generalized degeneracy or singularity; we refer to \cite{APPT22, AS23, BBLL22a, BBLL22b, BPRT20, dSJRR21, FRZ21}. 

We now display various approaches to achieve the H\"older estimate of the gradient in literature, and then outline the strategy of proof of our main theorem. In the elliptic setting, Imbert and Silvestre \cite{IS13} developed the interior $C^{1, \alpha}$-regularity for a viscosity solution $u$ of \eqref{eq:IS13} with $\gamma \geq 0$ and $f$ is bounded and continuous. In short, they first approximated $u$ by $u_k$ which is a viscosity solution of some modified equations with source terms $\|f_k\|_{L^{\infty}(B_1)} \to 0$. Then by taking a limit and applying `cutting lemma', they discovered that the limit function $u_{\infty}$ satisfies a homogeneous equation which guarantees the desired regularity. Nevertheless, the existence of the time derivative term $u_t$, which is not necessarily bounded in $L^{\infty}$-norm, prevents us from adopting the same idea to the parabolic equation \eqref{eq:model}. 


In the parabolic setting, Imbert, Jin, and Silvestre \cite{IJS19} proved interior $C^{1, \alpha}$-regularity for a viscosity solution $u$ of degenerate/singular quasilinear parabolic equations \eqref{eq:IJS19} with $p>1$ and $\gamma>-1$. The key ingredients of their proof were the existence of a smooth solution $u^{\varepsilon}$ that satisfies \eqref{eq:IJS19_reg} with prescribed boundary data, and the uniform $C^{1,\alpha}$-estimates for $u^{\varepsilon}$. Even though their approach for uniform estimates is also available for degenerate/singular fully nonlinear parabolic equations, the major challenge arises from the fact that the solvability of the Dirichlet problem \eqref{eq:mean} is still unknown. 

In view of Schuader fixed point theorem, the solvability of Dirichlet problems essentially follows from a priori estimate for associated problems. In the quasilinear setting, the key step in Bernstein technique for a priori estimate can be described as follows: the quantity $v\coloneqq |Du|^2$ can be regarded as a weak solution of a linear equation in divergence form, which leads to the weakened regularity assumption on $u$ from $C^3$ to $C^2$; see \cite{GT01, LSU68, Lie96} for details. However, such weak formulation strongly exploits the quasilinear structure of corresponding equations, and so it cannot be applied to our problem \eqref{eq:mean} which exhibits the fully nonlinear character. 

To overcome such difficulty, we develop gradient estimates by employing a modified version of Bernstein technique. To be precise, in order to avoid differentiating $u$ three times, we approximate the quantity $v = |Du|^2$ in terms of the difference quotient:
\begin{align*}
v^h(x, t) \coloneqq \sum_{k=1}^n \left(\frac{u(x+he_k, t)-u(x,t)}{h}\right)^2.
\end{align*}
Then it turns out that $v^h$ is a subsolution of certain parabolic equations in the classical sense, and so $v$ becomes a subsolution of similar parabolic equations in the viscosity sense, by passing the limit together with the stability theorem. In this viscosity formulation, the stronger $C^{2,\beta}$-regularity assumption must be imposed on $u$ rather than $C^2$, but it remains valid for our purpose to deduce the solvability of \eqref{eq:mean}.


This paper is organized as follows. \Cref{sec:preliminaries} consists of several notations, definitions, and auxiliary results for our main theorem. \Cref{sec:mean} is devoted to the proof of \Cref{thm:solv} based on a priori estimates. In \Cref{sec:approx}, we derive the uniform estimates for approximated solutions and then prove our main theorem, \Cref{thm:C1alpha}.
%
%
\section{Preliminaries}\label{sec:preliminaries}
\subsection{Notations}
We summarize some basic notations as follows.
\begin{enumerate}[label=(\roman*)]
\item Points: For $x=(x_1, \cdots,x_n) \in \mathbb{R}^n$, we denote 
$$x'=(x_1, \cdots,x_{n-1}) \in \mathbb{R}^{n-1}, ~ X=(x,t) \in \mathbb{R}^{n+1},   \text{ and }  O=(0,\cdots,0)\in \mathbb{R}^{n + 1}. $$ 
\item Sets: For a point $Y=(y,s) \in \mathbb{R}^{n+1}$ and $r>0$, we denote the cylinder as
\begin{align*}
	Q_r(Y) &= \{x \in  \mathbb{R}^n :|x-y|< r \}  \times (s -r^2 , s]. 
\end{align*}
Moreover, we define the bottom, corner, side, and parabolic boundary as
\begin{align*}
	\partial_b Q_r(Y) & = \{x \in  \mathbb{R}^n :|x-y|< r \}  \times \{ t= s-r^2 \}, \\
	\partial_c Q_r(Y) & = \{x \in  \mathbb{R}^n :|x-y| = r \}  \times \{ t= s-r^2 \}, \\
	\partial_s Q_r(Y) & =  \{x \in  \mathbb{R}^n :|x-y|= r \}  \times (s -r^2 , s), \\
	\partial_p Q_r(Y) & = \partial_b Q_r(y,s) \cup \partial_c Q_r(y,s) \cup \partial_s Q_r(y,s). 
\end{align*}
For convenience, we denote $Q_r = Q_r(O)$. 
\item Pucci's operators: Given ellipticity constants $0<\lambda \leq \Lambda$, we denote \textit{Pucci's operators} as follows: for $M \in \mathcal{S}^n \coloneqq \{ M : \text{$M$ is a $n\times n$ real symmetric matrix\}}$,
	\begin{align*}
		\mathcal{M}_{\lambda, \Lambda}^{+}(M) = \mathcal{M}^{+}(M) & \coloneqq \sup_{\lambda I \leq A \leq \Lambda I} \text{tr} (AM) ,\\
		\mathcal{M}^{-}_{\lambda, \Lambda}(M)  = \mathcal{M}^{-}(M) & \coloneqq \inf_{\lambda I \leq A \leq \Lambda I} \text{tr} (AM).
	\end{align*}

\item Distance functions:  Let $\Omega \subset \mathbb{R}^{n+1}$ be an open set and the parabolic distance function $d : \overline{\Omega} \times \overline{\Omega} \to [0,\infty) $ from $X=(x,t)$ to $Y=(y,s)$ is given by
\begin{equation*}
	d(X,Y) = \max \left\{ |x-y|, \sqrt{|t-s|}  \right\} .
\end{equation*}
\item Partial derivatives:
	We denote partial derivatives of $u$ as subscriptions.
\begin{equation*}
	u_t = \partial_t u = \frac{\partial u}{\partial t} , \quad  D_iu =\frac{\partial u}{\partial x_i} , \quad \text{and} \quad  D_{ij} u = \frac{\partial^2 u}{\partial x_i \partial x_j} .
\end{equation*}
\item Hölder spaces: Let $\Omega \subset \mathbb{R}^{n+1}$ be an open set and $\alpha \in (0,1)$.
\begin{itemize}
	\item $u \in C^{\alpha}(\overline{\Omega})$ means that there exists $C>0$ such that 
$$|u(X)-u(Y)| \leq Cd(X,Y)^{\alpha} \quad \text{for all } X, Y \in \Omega.$$
In other words, $u$ is $\frac{\alpha}{2}$-H\"older continuous in $t$ and  $\alpha$-H\"older continuous in $x$.
	\item $u \in C^{1,\alpha}(\overline{\Omega})$ means that $u$ is $\frac{\alpha + 1}{2}$-H\"older continuous in $t$ and $Du$ is $\alpha$-H\"older continuous in $x$.
	\item $u \in C^{2,\alpha}(\overline{\Omega})$ means that $u_t$ is $\frac{\alpha}{2}$-H\"older continuous in $t$ and $D^2u$ is $\alpha$-H\"older continuous in $x$.
\end{itemize}
\end{enumerate}
\subsection{Hypotheses on $F$}
We assume that the fully nonlinear operator $F:\mathcal{S}^n \to \mathbb{R}$ satisfies the following conditions:
\begin{enumerate} [label=\text{(F\arabic*)}]
\item \label{F1} $F$ is uniformly elliptic with $F(0)=0$; that is,  there exist constants $0<\lambda \leq \Lambda$ such that for any $M,N \in \mathcal{S}^n$, we have 
	$$\mathcal{M}^{-}(M-N) \leq F(M) - F(N) \leq \mathcal{M}^{+}(M-N).$$ 
\item \label{F2} $F$ is convex.
\item \label{F3} $F \in C^{1,1}(\mathcal{S}^n)$.
\end{enumerate}
\subsection{Viscosity solutions and comparison principles}
\begin{definition} [Test functions]
	Let $u$ be a continuous function in $Q_1 $. The function $\varphi : Q_1 \to \mathbb{R}$ is called \textit{test function} if it is $C^1$ with respect to $t$ and $C^2$ with respect to $x$.
	\begin{enumerate} [label=(\roman*)]
		\item We say that the test function $\varphi$ touches $u$ from above at $(x,t)$ if there exists an open neighborhood $U$ of $(x,t)$ such that 
		$$u \leq \varphi  \quad \mbox{in } U \qquad  \mbox{and} \qquad u(x,t) = \varphi(x,t). $$
		\item We say that the test function $\varphi$ touches $u$ from below at $(x,t)$ if there exists an open neighborhood $U$ of $(x,t)$ such that 
		$$u \ge \varphi  \quad \mbox{in } U \qquad  \mbox{and} \qquad u(x,t) = \varphi(x,t). $$
	\end{enumerate}
\end{definition}

\begin{definition} [Viscosity solutions]\label{vis1}
 Let $u$ be a function defined in $Q_1$.
 \begin{enumerate} [label=(\roman*)]
 \item Let $u$ be a upper semicontinuous function in $Q_1$. $u$ is called a \textit{viscosity subsoution} of \eqref{eq:model} in $Q_1 $ when the following condition holds: if for any $(x,t) \in Q_1 $ and any test function $\varphi$ touching $u$ from above at $(x,t)$, then
\begin{equation*}
	\varphi_t (x,t) \leq |D\varphi(x,t)|^{\gamma} F ( D^2 \varphi (x,t) ).
\end{equation*}
 \item Let $u$ be a lower semicontinuous function in $Q_1$. $u$ is called a \textit{viscosity supersoution} of \eqref{eq:model} in $Q_1 $ when the following condition holds: if for any $(x,t) \in Q_1 $ and any test function $\varphi$ touching $u$ from below at $(x,t)$, then
\begin{equation*}
	\varphi_t (x,t) \ge |D\varphi(x,t)|^{\gamma} F ( D^2 \varphi (x,t) ).
\end{equation*}
\end{enumerate}
Note that this definition of viscosity solutions can be extended to other fully nonlinear parabolic equations in a natural way.
\end{definition}

We now introduce a concept of parabolic semijets suggested in \cite[Section 8]{CIL92}.
\begin{definition}[Parabolic semijets]\label{semijet}
	 Let $u$ be a function defined in $Q_1$ and let $(x, t)\in Q_1$.
	\begin{enumerate}[label=(\roman*)]
		\item A \textit{parabolic superjet} $\mathcal{P}^{2, +}u(x, t)$ consists of $(a, p, M) \in \mathbb{R} \times \mathbb{R}^n \times \mathcal{S}^n$ which satisfy
		\begin{align*}
			u(y, s) &\leq u(x,t)+a(s-t)+\langle p, y-x \rangle \\
			&\qquad +\frac{1}{2} \langle M(y-x), y-x \rangle+o(|s-t|+|z-x|^2) \quad \text{as $(y, s) \to (x, t)$}. 
		\end{align*}
		Similarly, we can define a \textit{parabolic subjet} $\mathcal{P}^{2, -}u(x, t)$. It immediately follows that
		\begin{align*}
			\mathcal{P}^{2, -}u(x, t)=-\mathcal{P}^{2, +}(-u)(x, t).
		\end{align*} 
	
	\item A \textit{limiting superjet} $\overline{\mathcal{P}}^{2, +}u(x, t)$ consists of $(a, p, M) \in \mathbb{R} \times \mathbb{R}^n \times \mathcal{S}^n$ such that
		\begin{align*}
		\exists (x_n, t_n, a_n, p_n, M_n)  \ \text{with $(a_n, p_n, M_n) \in \mathcal{P}^{2, +}u(x_n, t_n)$ and } \\
		\text{$(x_n, t_n, u(x_n, t_n), a_n, p_n, M_n) \to (x, t, u(x, t), a, p, M)$} \quad \text{as } n \to \infty.
	\end{align*}
	We define a \textit{parabolic subjet} $\overline{\mathcal{P}}^{2, -}u(x, t)$ in a similar way.
	\end{enumerate}	
\end{definition}

Then we characterize viscosity sub/supersolutions in terms of parabolic semijets.
\begin{lemma}\label{lem:semijet}
	Let $u$ be a upper semicontinuous function in $Q_1$. Then $u$ is a viscosity subsolution of \eqref{eq:model} if and only if
	\begin{align*}
		a \leq |p|^{\gamma} F(M) \quad \text{for $(x, t) \in Q_1$ and $(a, p, M) \in \mathcal{P}^{2, +}u(x,t)$.}
	\end{align*}
\end{lemma}

We now consider a general fully nonlinear parabolic equation
\begin{equation}\label{eq:general}
		u_t=G(x, t, u, Du, D^2u),
\end{equation}
for an operator $G=G(x ,t , r, p, M) : B_1 \times [-1, 0) \times \mathbb{R} \times \mathbb{R}^n \times \mathcal{S}^n \to \mathbb{R}$.  We say $G$ is \textit{proper} if
\begin{align*}
	G(x, t, r, p, M) \leq G(x, t, s, p, N) \quad \text{whenever $s \leq r$ and $M \leq N$.}
\end{align*}
Moreover, we say $G$ satisfies the assumption 
\begin{enumerate}  [label=\text{(H)}]
\item \label{H} If there exists a function $\omega: [0, \infty] \to [0, \infty]$ with $\omega(0+)=0$ such that
\begin{align*} 
	G(x,t,r, \alpha(x-y), M) -G(y,t, r, \alpha(x-y), N) \leq \omega(\alpha|x-y|^2+|x-y|),
\end{align*}
whenever $x, y \in B_1$, $t \in [-1,0)$, $r \in \mathbb{R}$, $M, N \in \mathcal{S}^n$, and the following matrix inequality holds:
\begin{align*}
	-3\delta
	\begin{pmatrix}
		I & 0 \\
		0 & I
	\end{pmatrix}
\leq
	\begin{pmatrix}
		M & 0 \\
		0 & -N
	\end{pmatrix}
	\leq 3\delta
	\begin{pmatrix}
		I & -I \\
		-I & I
	\end{pmatrix}
.
\end{align*}
\end{enumerate}

\begin{theorem}[Comparison principle I, {\cite[Theorem 8.3]{CIL92}}]\label{thm:comp1}
		Let $G$ be continuous, proper, and satisfy the assumption \ref{H} for each fixed $t \in [-1, 0)$, with the same function $\omega$. Suppose that $u$ is a subsolution of \eqref{eq:general} and $v$ is a supersolution of \eqref{eq:general}. If $\limsup_{(y, s) \to (x, t)}u(y,s) \leq \liminf_{(y, s) \to (x, t)}v(y,s)$ for any  $(x, t) \in \partial_p Q_1$,  then $u \leq v$ in $Q_1$.
\end{theorem}

We remark that two operators $G_1(x,t, r, p, M)=G_1(p, M) \coloneqq (1+|p|^2)^{\gamma/2}F(M)$ and $G_2(x, t, r, p, M)= G_2(x,t, M) \coloneqq (1+\theta(x,t)^2)^{\gamma/2}F(M)$ (with H\"older continuous $\theta$) satisfy all assumptions for \Cref{thm:comp1}. However, the operator $G_3(x,t, r, p, M)=G_3(p, M) \coloneqq |p|^{\gamma}F(M)$ (with $\gamma<0$) is not continuous at $p=0$, and so, for \eqref{eq:model}, we require alternative version of comparison principles.

	\begin{theorem}[Comparison principle II, {\cite[Theorem 1]{Dem11}}]\label{thm:comp2}
		Suppose that $u$ is a subsolution of \eqref{eq:model} and $v$ is a supersolution of \eqref{eq:model}. If $\limsup_{(y, s) \to (x, t)}u(y,s) \leq \liminf_{(y, s) \to (x, t)}v(y,s)$ for any $(x, t) \in \partial_p Q_1$, then $u \leq v$ in $Q_1$.
	\end{theorem}
%
%
\section{The solvability of fully nonlinear mean curvature-type Dirichlet problems}\label{sec:mean}
%
%
\subsection{Reduction of the proof of \Cref{thm:solv}}
In this subsection, we suggest a reduction scheme for proving \Cref{thm:solv} in view of Schauder fixed point theorem and a priori estimate. We refer to \cite{Don91, LSU68, Lie96} for similar results in the quasilinear parabolic setting. We begin with two key lemmas.

\begin{lemma}[Schauder fixed point theorem]\label{lem:schauder}
	Let $T$ be a continuous and compact mapping of a Banach space $\mathcal{B}$ into itself, and suppose there exists a constant $L>0$ such that
	\begin{equation*}
		\|u\|_{\mathcal{B}} \leq L
	\end{equation*}
for all $u \in \mathcal{B}$ and $\sigma \in [0,1]$ satisfying $u=\sigma Tu$. Then $T$ has a fixed point.
\end{lemma}
\begin{proof}
	See \cite[Theorem 11.3]{GT01}.
\end{proof}

\begin{lemma}\label{lem:fixed}
	Suppose that  $\gamma\in\mathbb{R}$, and $F$ satisfies \ref{F1} and \ref{F2}. Let $v \in C^{1, \alpha}(\overline{Q_1})$ and $\varphi \in C^{2, \beta}(\overline{Q_1})$ for some $\alpha, \beta \in (0,1)$, and $\varphi$ satisfy the compatibility condition:
	\begin{equation*}
		\varphi_t=(1+|D\varphi|^2)^{\gamma/2}F(D^2 \varphi) \quad \text{on $\partial_c Q_1$}.
	\end{equation*}
	Then the  Dirichlet problem
	\begin{equation}\label{eq:fixed}
		\left\{ \begin{aligned}
			u_t&= (1+|Dv|^2 )^{\gamma/2}F(D^2u) &&\text{in $Q_1$} \\
			u&=\varphi &&\text{on $\partial_p Q_1$}
		\end{aligned}\right.
	\end{equation}
admits a unique classical solution $u$ which belongs to $C^{2, \overline{\alpha}}(\overline{Q_1})$ for some $\overline{\alpha} \in (0,1)$. Moreover, we have a uniform estimate
\begin{equation*}
	\|u\|_{C^{2, \overline{\alpha}}(\overline{Q_1})} \leq C,
\end{equation*}
where $C>0$ is a constant depending only on $n$, $\lambda$, $\Lambda$, $\gamma$, $\|\varphi\|_{C^{2, \beta}(\overline{Q_1})}$, and  $\|v\|_{C^{1, \alpha}(\overline{Q_1})}$.
\end{lemma}

\begin{proof}
	Let $G(x,t, M) \coloneqq (1+|Dv(x,t)|^2)^{\gamma/2}F(M)$. Since $v \in C^{1, \alpha}(\overline{Q_1})$, the operator $G$ satisfies the assumption given in \cite[Theorem 4.5]{CKLS99} or \cite[Theorem 8.4]{CKS00}. In other words, the standard Perron's method guarantees the existence of a viscosity solution $u \in C(\overline{Q_1})$ of \eqref{eq:fixed}. Furthermore, since $\varphi$ satisfies the compatibility condition at corner points, \cite[Theorem 3.1 and Theorem 3.2]{Wan92b} yields that $u \in C^{2, \overline{\alpha}}(\overline{Q_1})$ with a uniform estimate:
	\begin{equation*}
		\|u\|_{C^{2, \overline{\alpha}}(\overline{Q_1})} \leq  C,
	\end{equation*}
where $C>0$ is a constant depending only on $n$, $\lambda$, $\Lambda$, $\gamma$, $\|\varphi\|_{C^{2, \beta}(\overline{Q_1})}$, and  $\|v\|_{C^{1, \alpha}(\overline{Q_1})}$.

Finally, the uniqueness follows from the comparison principle, \Cref{thm:comp1}.
\end{proof}

We are now ready to provide the reduction of the proof of \Cref{thm:solv}.
\begin{lemma}\label{lem:reduction}
	Suppose that  $\gamma\in\mathbb{R}$, and $F$ satisfies \ref{F1} and \ref{F2}. Let $\varphi \in C^{2, \beta}(\overline{Q_1})$ for some $\beta \in (1/2, 1)$, and $\varphi$ satisfy the compatibility condition:
	\begin{align*}
		\varphi_t=(1+|D\varphi|^2)^{\gamma/2}F(D^2 \varphi) \quad \text{on $\partial_c Q_1$}.
	\end{align*}
	  Moreover, suppose that there exist constants $\alpha \in (0,1)$ and $L>0$ (which are independent of $u$ and $\sigma$) such that every $C^{2, \beta}(\overline{Q_1})$ solution $u$ of the $\sigma$-Dirichlet problems, 
	\begin{equation}\label{eq:sigma}
		\left\{ \begin{aligned}
			u_t&= (1+|Du|^2 )^{\gamma/2}F_{\sigma}(D^2u) &&\text{in $Q_1$} \\
			u&=\sigma\varphi &&\text{on $\partial_p Q_1$} 
		\end{aligned}\right.
	\end{equation}
	for $F_{\sigma}(M)\coloneqq\sigma F(\sigma^{-1}M)$ and $0 \leq \sigma \leq 1$, satisfies 
	\begin{align*}
		\|u\|_{C^{1, \alpha}(\overline{Q_1})} \leq L.
	\end{align*}
	Then the Dirichlet problem \eqref{eq:mean} is uniquely solvable in $C^{2, \beta}(\overline{Q_1})$.
\end{lemma}

\begin{proof}
	The uniqueness follows from the comparison principle, \Cref{thm:comp1}. For the existence, we fix an exponent $\alpha \in (0,1)$ and set a Banach space $\mathcal{B}=C^{1, \alpha}(\overline{Q_1})$.  Then \Cref{lem:fixed} implies that, for all $v \in \mathcal{B}$, the Dirichlet problem \eqref{eq:fixed} admits a unique classical solution $u \in \mathcal{B}$. Thus, we define a operator $T : \mathcal{B} \to \mathcal{B}$ by the relation $u=Tv$. It is easy to check that $u=\sigma Tu$ is corresponding to the $\sigma$-Dirichlet problems, \eqref{eq:sigma}.
	
	To apply \Cref{lem:schauder}, we claim that $T$ is continuous and compact. First of all, by the uniform estimate obtained in \Cref{lem:fixed}, we know that $T$ maps a bounded set in $\mathcal{B}=C^{1, \alpha}(\overline{Q_1})$ into a bounded set in $C^{2, \overline{\alpha}}(\overline{Q_1})$. By Arzela-Ascoli theorem, $C^{2, \overline{\alpha}}(\overline{Q_1})$ is precompact in $C^{2}(\overline{Q_1})$ and so in $\mathcal{B}$, which indicates the compactness of $T$. Moreover, suppose that a sequence of functions $\{v_m\}$ converges to $v$ in $\mathcal{B}$. Since $\{Tv_m\}$ is precompact in $C^2(\overline{Q_1})$, we suppose that a subsequence of $\{Tv_m\}$ converges to a limit function $u_0$ in $C^{2}(\overline{Q_1})$. By the definition of $T$, we note that
		\begin{equation*}
		\left\{ \begin{aligned}
			(Tv_m)_t&= (1+|Dv_m|^2 )^{\gamma/2}F(D^2(Tv_m)) &&\text{in $Q_1$}\\
			Tv_m&=\varphi &&\text{on $\partial_p Q_1$}
		\end{aligned}\right.
	\end{equation*}
and
	\begin{equation*}
	\left\{ \begin{aligned}
		(Tv)_t&= (1+|Dv|^2 )^{\gamma/2}F(D^2(Tv)) &&\text{in $Q_1$}\\
		Tv&=\varphi &&\text{on $\partial_p Q_1$}.
	\end{aligned}\right.
\end{equation*}
Letting $m \to \infty$, we conclude from the comparison principle, \Cref{thm:comp1}, that $\lim_{m \to \infty}Tv_m=u_0=Tv$.
\end{proof}

In view of \Cref{lem:reduction}, it only remains to derive a priori $C^{1, \alpha}$-estimate for $\sigma$-Dirichlet problems. More precisely, the desired a priori estimate will be deduced from the following five steps:
\begin{enumerate}[label=\Roman*.]
	\item An estimate of $\sup_{Q_1}|u|$, \Cref{lem:max_prin};
	\item An estimate of $\sup_{\partial_pQ_1}|Du|$, \Cref{lem:bdry_Du};
	\item An estimate of $\sup_{Q_1}|Du|$, \Cref{globalgradient};
	\item An estimate of $[Du]_{C^{\alpha}(Q')}$, \Cref{interiorholder};
	\item An estimate of $[Du]_{C^{\alpha}(\overline{Q_1})}$, \Cref{lem:gbd_c1a},
\end{enumerate}
where $Q' \subset \joinrel \subset Q_1$.
\subsection{A priori $L^{\infty}$ and gradient estimates}
\begin{lemma}[$L^{\infty}$-estimate] \label{lem:max_prin}
	Suppose that  $\gamma \in \mathbb{R}$ and $F$ satisfies \ref{F1}. Let $u \in C^{2}(\overline{Q_1})$ be a classical solution of \eqref{eq:mean} with $\varphi \in C(\partial_p Q_1)$. Then we have
\begin{align*}
	\sup_{Q_1}|u| \leq \sup_{\partial_pQ_1}|\varphi|.
\end{align*}
\end{lemma}

\begin{proof}
	Since $u \in C^2(\overline{Q_1})$, we observe that
	\begin{align*}
		u_t \geq \mathcal{M}^-_{\lambda', \Lambda'}(D^2u) \quad \text{in $Q_1$},
	\end{align*}
for the ellipticity constants $0<\lambda' \leq  \Lambda'$ which depend only on $\lambda$, $\Lambda$, $\gamma$, and $\|Du\|_{L^{\infty}(Q_1)}$. Then by applying the Alexandroff-Bakelman-Pucci estimate (see \cite[Theorem 3.14]{Wan92a}), we have
\begin{align*}
	u \geq -\sup_{\partial_p Q_1}|\varphi| \quad \text{in $Q_1$}.
\end{align*}
A similar argument gives the upper bound for $u$.
\end{proof}
\begin{lemma}[Boundary gradient estimates] \label{lem:bdry_Du}
Suppose that  $ \gamma > -2$ and $F$ satisfies \ref{F1}. Let $u \in C^2(Q_1) \cap C^1(\overline{Q_1})$ be a solution of \eqref{eq:mean} with $\varphi \in C^2(\overline{Q_1}) $. Then we have
\begin{equation*}
	\sup_{\partial_p Q_1}|Du| \leq C ,
\end{equation*}
where $C$ is a constant depending only on $n$, $\lambda$, $\Lambda$, $\gamma$, $\|u\|_{L^{\infty}(Q_1)}$, and $\|\varphi\|_{C^2(\overline{Q_1})}$. 
\end{lemma}

\begin{proof}
Since $Du = D \varphi$ on $\partial_b Q_1$, it is enough to prove the gradient estimates on $\partial_s Q_1$. Let $X_0=(x_0,t_0)  \in \partial_s Q_1$. Then there exists a point $y \in \mathbb{R}^n$ such that $|x_0-y|=1$ and $\{X \in \mathbb{R}^{n+1}: |x-y|<1\} \cap Q_1$ is empty.

Let $m=\sup_{Q_1}|u-\varphi|$ and $\Omega = \{X \in \mathbb{R}^{n+1} : 1< |x-y|< 1+B(e^{Am}-1)/A \}$. We define the barrier functions
\begin{equation*}
	w^{\pm}(x,t) \coloneqq \varphi(x,t) \pm f(d(x)), 
\end{equation*}
where
\begin{align*}
	f(r) \coloneqq \frac{1}{A} \log \left( 1 + \frac{A}{B}r\right)
	\quad \text{and} \quad
	d(x) \coloneqq |x-y|-1,
\end{align*}
for constants $A, B>0$ to be determined later. Then we have
$$D_i w^{\pm}(x,t) = D_i \varphi(x,t) \pm f'(d(x)) \cdot \frac{x_i-y_i}{|x-y|}. $$
We now choose $B^{-1}= 2e^{(2+\frac{2}{\gamma+2}) Am} \|\varphi\|_{C^2(\overline{Q_1})}$. Then since
$$f'(r) = \frac{1}{B+Ar} > \frac{1}{Be^{Am}}\ge 2 \sup_{Q_1}|D\varphi| \quad \text{for all } r \in \left(0, \frac{B}{A}(e^{Am}-1) \right),$$
we obtain
\begin{align*}
	|Dw^{\pm}| & \leq |D\varphi| + f'(d(x)) \leq 2  f'(d(x)) < \frac{2}{B} \quad \text{in } \Omega \cap Q_1.
\end{align*}
 Note that
\begin{equation*}
	D_i d(x)  = \frac{x_i-y_i}{|x-y|}, \qquad 
	D_{ij} d (x)  = \frac{\delta_{ij} }{|x-y|} - \frac{(x_i-y_i)(x_j-y_j)}{|x-y|^3},
\end{equation*}
and
\begin{equation*}
	D_{ij} w^+(x,t) = D_{ij} \varphi (x,t) + \frac{D_{ij} d(x)}{B+Ad(x)} -  \frac{A D_i d(x) D_j d(x)}{(B+Ad(x))^2}. 
\end{equation*}
Since the only nonzero eigenvalue of $e \otimes e$ is $|e|^2$ for $e\in\mathbb{R}^n \setminus \{0\}$, we have 
\begin{align*}
	\mathcal{M}^+ (Dd \otimes D d)= \Lambda \quad \text{and} \quad \mathcal{M}^- (Dd \otimes D d)=\lambda.
\end{align*}
Here $x \otimes y $ denotes the tensor product of $x,y \in \mathbb{R}^n$. Then we can choose sufficiently large $A>0$ depending only on $n$, $\lambda$, $\Lambda$, $\gamma$, and $\|\varphi\|_{C^2(\overline{Q_1})}$ so that
\begin{align*}
	F(D^2 w^+)&\leq \mathcal{M}^+ (D^2 \varphi) + \frac{1}{B+Ad(x)}  \mathcal{M}^+ (D^2 d) - \frac{A}{(B+Ad(x))^2}  \mathcal{M}^- (Dd \otimes D d) \\
	&\leq C \|\varphi\|_{C^2(\overline{Q_1})}  + \frac{n \Lambda -\lambda}{B}  - \frac{\lambda A}{B^2 e^{2Am}} \\
	&< - 3 \lambda A e^{(2+\frac{4}{\gamma+2})Am} \|\varphi\|_{C^2(\overline{Q_1})}^2
\end{align*}
in $\Omega \cap Q_1$. If $\gamma \ge 0$, then $ (1+|Dw^+|^2 )^{\gamma/2} \ge 1$ and so we can choose $A>0$ large enough to obtain
\begin{equation*}
	w_t^+ -  (1+|Dw^+|^2 )^{\gamma/2} F (D^2w^+)  \ge \varphi_t + 3 \lambda A e^{(2+\frac{4}{\gamma+2})Am} \|\varphi\|_{C^2(\overline{Q_1})}^2  >0 \quad \text{in $\Omega \cap Q_1$.}
\end{equation*}
If $-2< \gamma < 0$, then $(1+|Dw^+|^2 )^{\gamma/2} \ge 3^{\gamma} B^{-\gamma}$ provided that $B \leq \sqrt{5}$. We again choose $A>0$ large enough to obtain
\begin{align*}
	w_t^+ -  (1+|Dw^+|^2 )^{\gamma/2} F (D^2w^+)  
	\geq \varphi_t + 3\lambda 6^{\gamma} A e^{(2\gamma+4) Am}  \|\varphi\|_{C^2(\overline{Q_1})}^{2+\gamma}>0 \quad \text{in $\Omega \cap Q_1$. }
\end{align*}
Furthermore, the function $w^+$ satisfies
\begin{align*}
	w^+ &\ge \varphi = u \quad \text{on } \Omega \cap \partial_p Q_1, \\
	w^+ &= \varphi + m \ge u \quad \text{on } \partial_p \Omega \cap Q_1, \\
	w^+(X_0) &= \varphi(X_0)= u(X_0).
\end{align*}
Then the comparison principle \Cref{thm:comp1} implies that 
$w^+ \ge u$ in $\Omega \cap Q_1.$ Thus, we have
\begin{align*}
	u(X)-u(X_0)\leq w^+(X)-w^+(X_0)
	\leq \|\varphi\|_{C^1(\overline{Q_1})}|X-X_0| + \frac{1}{B}d(x) 
	\leq C |X-X_0| .
\end{align*}
for all $X \in \Omega \cap Q_1$. The lower bound follows from a similar argument for $w^-$. Since $u \in C^1(\overline{Q_1})$, we conclude that $|Du| \leq C$ on $\partial_p Q_1$.
\end{proof}

We now develop global gradient estimates based on the Bernstein technique in terms of difference quotient, as introduced in \Cref{sec:introduction}. We recall two quantities, namely, $v\coloneqq|Du|^2$ and its difference quotient counterpart
\begin{align*}
	v^h(x, t) \coloneqq \sum_{k=1}^n \left(\frac{u(x+he_k, t)-u(x,t)}{h}\right)^2.
\end{align*}

\begin{lemma}[Global gradient estimates]\label{globalgradient}
	Suppose that  $ \gamma \in \mathbb{R}$, and $F$ satisfies \ref{F1} and \ref{F3}. Let $\beta \in (1/2, 1)$ and let $u \in C^{2, \beta}(\overline{Q_1})$ satisfy 
	$$u_t=(1+|Du|^2)^{\gamma/2}F(D^2u) \quad \text{in } Q_1.$$ 
	Then we have 
	\begin{align*}
		\sup_{Q_1}|Du|=\sup_{\partial_p Q_1}|Du|.
	\end{align*}
\end{lemma}	

\begin{proof}
	Let $v \coloneqq |Du|^2$ and 
	\begin{align*}
		v^h(x,t)\coloneqq \sum_{k=1}^n \left(u_k^h(x,t)\right)^2,
	\end{align*}
	where we write
	\begin{align*}
		u_k^h(x, t)\coloneqq\frac{u(x+he_k, t)-u(x, t)}{h} \quad \text{for } k =1,\cdots, n.
	\end{align*}
	
	\noindent {\textbf{(Step 1: An equation satisfied by $v^h$.)}} Since $u \in C^{2, \beta}(\overline{Q_1})$, the following two equations hold in the classical sense:
	\begin{align*}
		u_t(x+he_k, t) &=  (1+|Du(x+he_k, t)|^2  )^{\gamma/2}F(D^2u(x+he_k, t)), \\
		u_t(x,t) &= (1+|Du(x, t)|^2  )^{\gamma/2}F(D^2u(x,t)).
	\end{align*} 
	By combining two equalities, we obtain
	\begin{align*}
		\partial_ t u_k^h  &=  (1+|Du(x, t)|^2  )^{\gamma/2} \frac{ F (D^2u(x+he_k, t) ) - F (D^{2}u(x,t) )}{h} \\
		&\qquad+F (D^{2}u(x+he_k,t) )  \frac{ (1+|Du(x+he_k, t)|^2  )^{\gamma/2} - (1+|Du(x, t)|^2  )^{\gamma/2} }{h}.
	\end{align*}
	Since $F$ satisfies \ref{F3}, we can apply the fundamental theorem of calculus to derive
	\begin{align*}
		&F (D^{2}u(x+he_k, t) )-F (D^{2} u(x, t) ) \\
		&\quad=\int_0^1 \frac{d}{ds} F  (sD^2u(x+he_k, t)+(1-s)D^2u(x, t) ) \, ds \\
		&\quad=\int_0^1 F_{ij} (sD^2u(x+he_k, t)+(1-s)D^2u(x, t) )   ( D_{ij}u(x+he_k, t) - D_{ij} u(x, t)  ) \, ds \\
		&\quad=h a_{ij}^{kh}D_{ij}u_k^h,
	\end{align*}
	where
	\begin{align*}
		a_{ij}^{kh}(x, t)\coloneqq\int_0^1 F_{ij} (sD^2u(x+he_k, t)+(1-s)D^2u(x, t) )\,ds.
	\end{align*}
	Note that since $F$ is uniformly elliptic with ellipticity constants $\lambda$ and $\Lambda$, we have
	\begin{align*}
		\lambda |\xi|^2 \leq F_{ij}(M) \xi_i\xi_j \leq \Lambda|\xi|^2 \quad \text{for any $M \in \mathcal{S}^n$ and $\xi \in \mathbb{R}^n$}, 
	\end{align*}
	which implies that 
	\begin{align}\label{aelliptic}
		\lambda |\xi|^2 \leq a_{ij}^{kh}(x, t) \xi_i\xi_j \leq \Lambda|\xi|^2 \quad \text{for any $(x, t) \in Q_1$ and $\xi \in \mathbb{R}^n$}.
	\end{align}
	In a similar way, we can calculate
	\begin{align*}
		& (1+|Du(x+he_k, t)|^2  )^{\gamma/2} - (1+|Du(x, t)|^2  )^{\gamma/2}\\
		&\quad=\int_0^1\gamma  (1+|\chi^{kh}(s)|^2 )^{\gamma/2-1} \chi^{kh}(s) \cdot  (Du (x+he_k, t)-Du(x, t) )\, ds\\
		&\quad=h\psi^{kh} \cdot Du_k^h,
	\end{align*}
	where 
	\begin{align*}
		[\chi^{kh}(s)](x, t)\coloneqq s Du(x+he_k, t)+(1-s) Du(x, t)
	\end{align*}
	and 
	\begin{align*}
		\psi^{kh}(x, t)\coloneqq\gamma \int_0^1 (1+|\chi^{kh}(s)|^2 )^{\gamma/2-1} \chi^{kh}(s) \,ds.
	\end{align*}
	Since $u \in C^{2, \beta}(\overline{Q_1})$, we have $|[\chi^{kh}(s)](x, t)| \leq 2\|Du \|_{L^{\infty}(Q_1)}$ and so
	\begin{align}\label{ybound}
		|\psi^{kh}(x, t)| \leq \Psi.
	\end{align}
	Here note that the constant $\Psi>0$ can depend only on $\gamma$ and $\|Du\|_{L^{\infty}(Q_1)}$. By combining previous computations, we arrive at
	\begin{align}\label{eq:1}
		\partial_t u_k^h =\left(1+|Du(x, t)|^2 \right)^{\gamma/2} a_{ij}^{kh}D_{ij}u_k^h+F (D^{2}u(x+he_k,t) )  \psi^{kh} \cdot Du_k^h.
	\end{align}
	Multiplying both sides of \eqref{eq:1} by $2u_k^h$ and summing the resulting equations from $k=1$ to $n$, we have
	\begin{equation}\label{eq:2}
		\begin{aligned}
			\partial_t v^h  =\sum_{k=1}^n (1+|Du(x, t)|^2  )^{\gamma/2} a_{ij}^{kh}D_{ij}(u_k^h)^2+\sum_{k=1}^nF (D^{2}u(x+he_k,t) )  \psi^{kh} \cdot D(u_k^h)^2\\
			-2\sum_{k=1}^n (1+|Du(x, t)|^2  )^{\gamma/2} a_{ij}^{kh}D_iu_k^hD_ju_k^h.
		\end{aligned}
	\end{equation}
	
	\noindent 
	{\textbf{(Step 2: An equation satisfied by $v$.)}} In this step, we take a limit $h \to 0$ and employ the standard stability argument to find an equation satisfied by $v$. For this purpose, we first observe that, $a_{ij}^{kh}$ and $\psi^{kh}$ can be approximated by $a_{ij}^0$ and $\psi^0$ which are independent of $h>0$; i.e.
	\begin{equation*}
		\lim_{h\to0}a_{ij}^{kh}=a_{ij}^0 \coloneqq  F_{ij} (D^2u(x, t) )
	\end{equation*}
and
 	\begin{equation*}
		\lim_{h \to 0}\psi^{kh} =\psi^0\coloneqq\gamma  (1+|Du(x,t)|^2 )^{\gamma/2-1}Du(x,t).
	\end{equation*}
	To be precise, since $F$ is $C^{1,1} $ and $u \in C^{2, \beta}(\overline{Q_1})$, we have
	\begin{align*}
		|a_{ij}^{kh}-a_{ij}^0|&=\left|\int_0^1 F_{ij} (sD^2u(x+he_k, t)+(1-s)D^2u(x, t) )\,ds-F_{ij} (D^2u(x, t) )\right|\\
		&\leq \int_0^1 \|F\|_{C^{1,1}(\mathcal{S}^n)} |D^2u(x+he_k, t)-D^2u(x,t)| s \, ds \\
		&\leq \|F\|_{C^{1,1}(\mathcal{S}^n)} \|u\|_{C^{2, \beta}(\overline{Q_1})} h^{\beta}.
	\end{align*}
In a similar way, we have
	\begin{equation*}
		|\psi^{kh}-\psi^0|\leq C(\gamma)\|u\|_{C^{1,1}(\overline{Q_1})}h
	\end{equation*}
and
\begin{equation*}
	|F(D^2u(x+he_k, t))-F(D^2u(x, t))| \leq \|F\|_{C^{0,1}(\mathcal{S}^n)} \|u\|_{C^{2, \beta}(\overline{Q_1})}  h^{\beta}.
\end{equation*}
	Moreover, we can check that for sufficiently small $h>0$,
	\begin{align*}
		|D_{ij}(u_k^h)^2| &\leq 2|u_k^hD_{ij}u_k^h|+2|D_iu_k^hD_ju_k^h| \\
		&\leq 2\|u\|_{C^{0,1}(\overline{Q_1})} \cdot \frac{|D_{ij}u(x+he_k, t)-D_{ij}u(x, t)|}{h}+2\|u\|_{C^{1,1}(\overline{Q_1})}^2\\
		& \leq 4\|u\|_{C^{2, \beta}(\overline{Q_1})}^2 h^{\beta-1}
	\end{align*}
and
\begin{align*}
	|D(u_k^h)^2|=2|u_k^hDu_k^h| \leq 2n \|u\|_{C^2(\overline{Q_1})}^2.
\end{align*}
	Therefore, recalling \eqref{ybound} and $\beta >1/2$, the equation \eqref{eq:2} can be written as
	\begin{align*}
		\partial_t v^h &\leq \sum_{k=1}^n (1+|Du(x, t)|^2  )^{\gamma/2} a_{ij}^{0}D_{ij}(u_k^h)^2 \\
		& \qquad +\sum_{k=1}^n F (D^{2}u(x,t) ) \psi^{0} \cdot D(u_k^h)^2-C|D^2u|^2+O(h^{2\beta-1}) \quad \text{as }  h \to 0 \\
		&= (1+|Du(x, t)|^2  )^{\gamma/2} a_{ij}^{0}D_{ij}v^h \\
		&\qquad +F (D^{2}u(x,t) )  \psi^{0} \cdot Dv^h- C|D^2u|^2+O(h^{2\beta-1}) \quad \text{as }  h \to 0 .
	\end{align*}
	Here we note that, by the ellipticity given in \eqref{aelliptic},
	\begin{align*}
		a_{ij}^{kh}D_iu_k^hD_ju_k^h \geq \lambda |Du_k^h|^2=\lambda|D^2u|^2-o(h) \quad \text{as }  h \to 0.
	\end{align*}
	Hence, by applying the Cauchy-Schwartz inequality,
	\begin{align*}
		\partial_t v^h &\leq  (1+|Du(x, t)|^2  )^{\gamma/2} a_{ij}^{0}D_{ij}v^h \\
		&\quad \quad \quad \;\;\, +F (D^{2}u(x,t) )  \psi^{0} \cdot Dv^h	-C|D^2u|^2+O(h^{2\beta-1}) \quad \text{as }  h \to 0\\
		&\leq \mathcal{M}_{\lambda', \Lambda'}^+(D^2v^h)+C|D^2u| |Dv^h|-C|D^2u|^2+O(h^{2\beta-1}) \quad \text{as }  h \to 0\\
		&\leq \mathcal{M}_{\lambda', \Lambda'}^+(D^2v^h)+C|Dv^h|^2+O(h^{2\beta-1}) \quad \text{as }  h \to 0,
	\end{align*}
	where the ellipticity constants $\lambda', \Lambda'$ depend only on $n$, $\lambda$, $\Lambda$, $\gamma$, and $\|Du\|_{L^{\infty}(Q_1)}$. Since $u \in C^{2, \beta}(\overline{Q_1})$ implies that $Dv^h \to Dv$ uniformly, letting $h \to 0$ together with the stability theorem (see \cite[Proposition 2.9]{CC95}) yields that
	\begin{align*}
		v_t \leq \mathcal{M}_{\lambda', \Lambda'}^+(D^2v)+C|Dv|^2 \quad \text{in the viscosity sense},
	\end{align*}
	where $C>0$ depends only on $n$, $\lambda$, $\Lambda$, $\gamma$, and $\|Du\|_{L^{\infty}(Q_1)}$.
	
\noindent {\textbf{(Step 3: Global gradient estimates.)}}
Set $V_0 \coloneqq \sup_{\partial_p Q_1} |Du|^2$. Then we can choose a universal constant $\mu>0$ such that if we let
\begin{align*}
	\overline{v} \coloneqq \frac{1}{\mu}\left(1-e^{\mu(v-V_0)}\right),
\end{align*}
then $\overline{v}$ satisfies
\begin{align*}
	\overline{v}_t \geq \mathcal{M}_{\lambda', \Lambda'}^-(D^2\overline{v}) \quad \text{in the viscosity sense}.
\end{align*}
By the Alexandroff-Bakelman-Pucci estimate, we have
\begin{align*}
	\overline{v} \geq \inf_{\partial_p Q_1} \overline{v}=0 \quad \text{in $Q_1$}.
\end{align*}
Therefore, we conclude that
\begin{align*}
	|Du|^2=v \leq V_0=\sup_{\partial_p Q_1} |Du|^2 \quad \text{in $Q_1$},
\end{align*}
which finishes the proof.
\end{proof}
\subsection{A priori H\"older estimates for the gradient}
We first develop an interior a priori H\"older estimate for the gradient, whose proof is similar to the one of \Cref{globalgradient}. However, the proof of \Cref{interiorholder} requires an additional effort to control the oscillation of $|Du|$ in an iterative way.
\begin{lemma}[Interior a priori H\"older estimate for the gradient] \label{interiorholder}
Suppose that  $\gamma \in \mathbb{R}$, and $F$ satisfies \ref{F1} and \ref{F3}. Let $\beta \in (1/2, 1)$ and let $u \in C^{2, \beta}(\overline{Q_1})$ satisfy 
$$u_t=(1+|Du|^2)^{\gamma/2}F(D^2u) \quad \text{in }Q_1.$$
Moreover, assume that $|u|+|Du| \leq K$ in $Q_1$, for some $K>0$. Then there exists a constant $\alpha \in (0,1)$ depending only on $n$, $\lambda$, $\Lambda$, $\gamma$, and $K$ such that for any $Q' \subset \joinrel \subset Q_1$, we have 
	\begin{align*}
		[Du]_{C^\alpha(Q')} \leq C d^{-\alpha},
	\end{align*}
where $d=\mathrm{dist}(Q', \partial_pQ_1)$ and $C$ is a constant depending only on $n$, $\lambda$, $\Lambda$, $\gamma$, and  $K$. In addition, we have
\begin{align*}
	\sum_{k=1}^n \osc_{Q_r(Y)} D_k u \leq C \left(\frac{r}{\rho}\right)^\alpha  \sum_{k=1}^n\osc_{Q_\rho(Y)} D_k u  \quad \text{for $0< r \leq \rho \leq \textnormal{dist}(Y, \partial_p Q_1)$}.
\end{align*}
\end{lemma}

\begin{proof}
	Let $w\coloneqq\delta D_lu+|Du|^2$ for $l=1, ..., n$, where $\delta>0$ to be determined. As in \Cref{globalgradient}, we approximate $w$ in sense of difference quotient:
	\begin{align*}
		w^h(x,t)\coloneqq \delta u_l^h(x,t)+ \sum_{k=1}^n \left(u_k^h(x,t)\right)^2,
	\end{align*}
where we write
\begin{align*}
	u_k^h(x, t)\coloneqq\frac{u(x+he_k, t)-u(x, t)}{h} \quad \text{for } k =1,\cdots, n.
\end{align*}
By following Step 1 and Step 2 in the proof of \Cref{globalgradient}, we obtain
\begin{align*}
	\partial_t w^h  \leq \mathcal{M}_{\lambda', \Lambda'}^+(D^2w^h)+C|Dw^h|^2+O(h^{2\beta-1}) \quad \text{as }  h \to 0,
\end{align*}
where the ellipticity constants $\lambda', \Lambda'$ depend only on $n$, $\lambda$, $\Lambda$, $\gamma$, and $K$. Since $u \in C^{2, \beta}(\overline{Q_1})$ implies that $Dw^h \to Dw$ uniformly, letting $h \to 0$ together with the stability theorem yields that
\begin{align*}
	w_t \leq \mathcal{M}_{\lambda', \Lambda'}^+(D^2w)+C|Dw|^2 \quad \text{in the viscosity sense},
\end{align*}
where $C>0$ depends only on $n$, $\lambda$, $\Lambda$, $\gamma$, and $K$.

We are now ready to derive uniform interior H\"older estimate for $Du$. We suppose that 
$$u_t=(1+|Du|^2)^{\gamma/2}F(D^2u) \quad \text{in }Q_{4r}(x_0,t_0).$$ 
For brevity, we write $Q_{4r}=Q_{4r}(x_0, t_0)$ and $Q_r=Q_r(x_0, t_0)$. We first choose the constant $\delta=10nK$ and choose $l \in \{1, ..., n\}$ so that 
\begin{align*}
	\osc_{Q_{4r}} D_ku \leq \osc_{Q_{4r}} D_lu \quad \text{for all $k \in \{1, .., n\}$}.
\end{align*}
We define
\begin{align*}
	w^{\pm}=w_k^{\pm} \coloneqq \pm \delta D_ku+|Du|^2 \quad \text{and} \quad W^{\pm}_k \coloneqq \sup_{Q_{4r}}w_k^{\pm}.
\end{align*} 
Then it is easy to check that
\begin{align*}
	\osc_{Q_{4r}} |Du|^2 \leq 2nM \osc_{Q_{4r}} D_lu
\end{align*}
and so
\begin{align*}
	8nM \osc_{Q_{4r}}D_ku \leq	\osc_{Q_{4r}}w^{\pm} \leq 12nM \osc_{Q_{4r}}D_ku.
\end{align*}
On the other hand, we note that
\begin{align*}
	\partial_t w^{\pm} \leq \mathcal{M}_{\lambda', \Lambda'}^{+}(D^2w^{\pm})+C|Dw^{\pm}|^2 \quad \text{in the viscosity sense}.
\end{align*}
We can choose a universal constant $\mu>0$ such that if we let
\begin{align*}
	\overline{w} \coloneqq \frac{1}{\mu}\left(1-e^{\mu(w^{\pm}-W^{\pm})}\right),
\end{align*}
then $\overline{w}$ satisfies
\begin{align*}
	\overline{w}_t \geq \mathcal{M}_{\lambda', \Lambda'}^-(D^2\overline{w}) \quad \text{in the viscosity sense}.
\end{align*}
Since $\overline{w}$ is a nonnegative supersolution, we can apply the weak Harnack inequality (see \cite[Theorem 6.18]{Lie96} or \cite[Corollary 4.14]{Wan92b}) to obtain
\begin{align*}
	 \left(\fint_{\Theta_r} \overline{w}^{\varepsilon_0} \,dx\,dt\right)^{1/\varepsilon_0}\leq c \, \inf_{Q_r}\overline{w} \quad \text{for some $\varepsilon_0>0$},
\end{align*}
where $\Theta_r \coloneqq Q_r(x_0, t_0-4r^2)$. Since
\begin{align*}
	c_1\, (W^{\pm}-w^{\pm}) \leq \overline{w} \leq W^{\pm}-w^{\pm},
\end{align*}
we have
\begin{align}\label{eq:comp1}
	\left(\fint_{\Theta_r} (W^{\pm}-w^{\pm})^{\varepsilon_0} \,dx\,dt\right)^{1/\varepsilon_0}\leq c \, \inf_{Q_r}(W^{\pm}-w^{\pm}).
\end{align}
Here we observe that
\begin{align}\label{eq:comp2}
	\inf_{Q_r}(W^{\pm}-w^{\pm})=\sup_{Q_{4r}}w^{\pm}-\sup_{Q_r}w^{\pm} \leq\osc_{Q_{4r}}w^{\pm}-\osc_{Q_r}w^{\pm}.
\end{align} 
Moreover, since $w_k^++w_k^-=2|Du|^2$, we have
\begin{align*}
	(W^+-w^+)+(W^--w^-) &\geq \delta \osc_{Q_{4r}}D_ku +2\left(\inf_{Q_{4r}}|Du|^2-|Du|^2\right)\\
	&\geq 6nM \osc_{Q_{4r}}D_ku \\
	&\geq \frac{1}{2}\osc_{Q_{4r}}w^{\pm} \quad \text{in $Q_{4r} (\supset \Theta_r)$}.
\end{align*}
Thus, the inequality
\begin{align}\label{eq:comp3}
	W^{\pm}-w^{\pm} \geq \frac{1}{4}\osc_{Q_{4r}}w^{\pm}
\end{align}
holds for either $w^{+}$ or $w^-$. Without loss of generality, we may assume it holds for $w^+$. By combining the estimate \eqref{eq:comp1}, \eqref{eq:comp2}, and \eqref{eq:comp3}, we arrive at
\begin{align*}
	\osc_{Q_r} w^+ \leq \frac{4c-1}{4c} \osc_{Q_{4r}} w^+.
\end{align*}
By following the standard iteration argument (see \cite[Lemma 13.5]{GT01}, for example), we conclude that there exists a constant $\alpha \in (0,1)$ depending only on $n$, $\lambda$, $\Lambda$, $\gamma$, and $K$ such that 
\begin{align*}
	\sum_{k=1}^n \osc_{Q_r(Y)} D_k u \leq C \left(\frac{r}{\rho}\right)^\alpha   \sum_{k=1}^n\osc_{Q_\rho(Y)} D_k u   \quad \text{for $0< r \leq \rho \leq \textnormal{dist}(Y, \partial_p Q_1)$}
\end{align*}
and so
\begin{align*}
	\osc_{Q_r}D_ku \leq Cr^{\alpha} \quad \text{for any $k =1,\cdots,n$},
\end{align*}
where $C>0$ is a constant depending only on $n$, $\lambda$, $\Lambda$, $\gamma$, and $K$.
\end{proof}

If $Y \in \partial_b Q_1 \cup \partial_c Q_1$, we understand $Q_r(Y)$ as $B_r(y) \times [-1,r^2-1)$.
\begin{lemma}[Pointwise Boundary $C^{1, \alpha}$-estimate] \label{lem:bd_c1a}
Suppose that  $\gamma \in \mathbb{R}$ and $F$ satisfies \ref{F1}. Let $u \in C^1(\overline{Q_1})$ be a viscosity solution of \eqref{eq:mean} with $\varphi\in C^{2} (\overline{Q_1})$. Then for each $Y\in \partial_p Q_1$, there exist constants $\alpha \in (0,1)$ depending only on $n$, $\lambda$, $\Lambda$, $\gamma$, and $\|\varphi\|_{C^2(\overline{Q_1})}$ and a vector-valued function $G \in C^{\alpha} (\partial_p Q_1 \cap Q_{1/2}(Y), \mathbb{R}^n)$  such that for any $X_0 \in \partial_p Q_1 \cap Q_{1/2}(Y)$, we have
\begin{equation*}
	|u(X) - u(X_0) - G(X_0)\cdot(x-x_0)| \leq Cd(X,X_0)^{1+\alpha} \quad \text{for all } X \in Q_1 \cap Q_1(Y),
\end{equation*}
where $C$ is a constant depending only on $n$, $\lambda$, $\Lambda$, $\gamma$, and $\|\varphi\|_{C^{2}(\overline{Q_1})}$.

\end{lemma}

\begin{proof}
Since $u \in C^1(\overline{Q_1})$, we observe that
\begin{align*}
		\mathcal{M}^-_{\lambda', \Lambda'}(D^2u) \leq u_t \leq \mathcal{M}^+_{\lambda', \Lambda'}(D^2u) \quad \text{in $Q_1$},
	\end{align*}
for the ellipticity constants $0<\lambda' \leq  \Lambda'$ which depend only on $n$, $\lambda$, $\Lambda$, $\gamma$, and $\|\varphi\|_{C^2(\overline{Q_1})} $ by \Cref{lem:max_prin}, \Cref{lem:bdry_Du}, and \Cref{globalgradient}. Then by applying the boundary $C^{1,\alpha}$-estimate for uniformly parabolic equations (see  \cite[Theorem 2.1, Theorem 2.11]{Wan92b}, \cite[Theorem 1.1]{SS14}, and \cite[Theorem 2.4]{ABV20}), we have the desired conclusion.
\end{proof}
By using the following lemma, we can obtain $C^{1,\alpha}$-regularity near the parabolic boundary by connecting interior $C^{1,\alpha_1}$-regularity with pointwise boundary $C^{1,\alpha_2}$-regularity. It can be found in \cite[Lemma 12.4]{Lie96}.
\begin{lemma} \label{lem:comb_int_bdry}
Let $u \in C^1(Q_R(Y)) \cap C(\overline{Q_R(Y)})$. Suppose that there exists constants $A\ge0$, $B\ge0$, and $\alpha \in (0,1)$ such that
\begin{equation} \label{eq:int_hol_Du}
	\sum_{k=1}^n\osc_{Q_r(Z)} D_k u \leq A \left(\frac{r}{\rho}\right)^\alpha \left(\sum_{k=1}^n \osc_{Q_\rho(Z)} D_k u + B \rho^\alpha \right)
\end{equation}
whenever $r\leq \rho$ and $Q_\rho (Z) \subset Q_R(Y)$. Then for any $a \in \mathbb{R}$ and $e \in \mathbb{R}^n$, we have
$$\sup_{Q_{R/2}(Y)} |Du-e| \leq C \left( \frac{1}{R} \sup_{Q_R(Y)} |u-a- e\cdot x | + B R^{\alpha}\right),$$
where $C$ is a constant depending only on $A$, $n$, and $\alpha$. 
\end{lemma}
Combining \Cref{interiorholder}, \Cref{lem:bd_c1a}, and \Cref{lem:comb_int_bdry}, we obtain $C^{1, \alpha}$-regularity near $\partial_p Q_1$.

\begin{lemma} [$C^{1, \alpha}$-estimate near $\partial_p Q_1$]\label{lem:bd_c1a_n}
Suppose that  $\gamma \in \mathbb{R}$, and $F$ satisfies \ref{F1} and \ref{F3}. Let $\beta \in (1/2, 1)$ and  let $u \in C^{2,\beta}(\overline{Q_1})$ be a solution of \eqref{eq:mean} with $\varphi\in C^{2} (\overline{Q_1})$. Then there exists constant $\alpha \in (0,1)$ depending only on $n$, $\lambda$, $\Lambda$, $\gamma$, and $\|\varphi\|_{C^2(\overline{Q_1})}$ such that for any $X_0 \in \partial_p Q_1$, $r \in (0,1/2)$, and $k=1,\cdots,n$, we have
\begin{equation*}
	\osc _{Q_1\cap Q_r(X_0) } D_k u \leq Cr^{\alpha} , 
\end{equation*}
where $C$ is a constant depending only on $n$, $\lambda$, $\Lambda$, $\gamma$, and $\|\varphi\|_{C^{2}(\overline{Q_1})}$.
\end{lemma}

\begin{proof}
First we consider the case $X_0 \in \partial_s Q_1$. For any fixed $Y \in Q_1 \cap Q_r(X_0) $ and $r \in (0,1/2)$, by \Cref{interiorholder}, we can see that there exists a constant $\alpha_1 \in (0,1)$ such that
\begin{equation*}  
	\sum_{k=1}^n\osc_{Q_r(Z)} D_k u \leq A \left(\frac{r}{\rho}\right)^{\alpha_1}  \sum_{k=1}^n \osc_{Q_\rho(Z)} D_k u 
\end{equation*}
whenever $r\leq \rho \leq \textnormal{dist}(Z,\partial_p Q_1)$, $Q_\rho (Z) \subset Q_R(Y)$, and $R= \text{dist}(Y,\partial_p Q_1) /2$. Hence by \Cref{lem:comb_int_bdry}, for any $a \in \mathbb{R}$ and $e \in \mathbb{R}^n$, we have
\begin{equation} \label{eq:sup_Du-e}
	\sup_{Q_{R/2}(Y)} |Du-e| \leq C R^{-1} \sup_{Q_R(Y)} |u-a- e\cdot x | .
\end{equation}

On the other hand, from \Cref{lem:bd_c1a},  there exist constants $\alpha_2 \in (0,1)$, $C>0$, and a vector-valued function $G=(g_1,\cdots,g_n) \in C^{\alpha_2}(\partial_p Q_1 \cap Q_{1/2}(\tilde{Y}), \mathbb{R}^n)$ such that 
\begin{equation*}
	|u(X) - u(\tilde{Y}) -G(\tilde{Y})\cdot(x-\tilde{y})| \leq Cd(X,\tilde{Y})^{1+\alpha_2} \quad \text{for all } X \in Q_1 \cap Q_1(\tilde{Y}), 
\end{equation*}
where $\tilde{Y} \in \partial_p Q_1$ with $d(Y,\tilde{Y}) =2R$. This implies that 
\begin{align*}
	 |u(X)- u(\tilde{Y}) - G(\tilde{Y})\cdot(x-\tilde{y}) | \leq Cd(X,\tilde{Y})^{1+\alpha_2} \leq CR^{1+\alpha_2}
\end{align*}
for all $X \in Q_R(Y)$. So if we take $a=u(\tilde{Y})- G(\tilde{Y})\cdot \tilde{y} $ and $e= G(\tilde{Y})$ in \eqref{eq:sup_Du-e}, we have
\begin{equation*} 
	 |Du(Y)-G(\tilde{Y})| \leq \sup_{Q_{R/2}(Y)} |Du-G(\tilde{Y})| \leq C d(Y,\tilde{Y})^{\alpha_2}.
\end{equation*}
Since $G=(g_1,\cdots,g_n) \in C^{\alpha_2}(\partial_p Q_1 \cap Q_{1/2}(\tilde{Y}), \mathbb{R}^n)$, we have
\begin{align*}
	D_k u(X) -D_k u(Y) & \leq |D_k u(X) -g_k(\tilde{X}) | +  |g_k (\tilde{X}) -g_k (\tilde{Y}) | +  |g_k (\tilde{Y}) -D_k u(Y) | \\
	& \leq Cd(X,\tilde{X}) ^{\alpha_2} + Cd(\tilde{X},\tilde{Y})^{\alpha_2} + C d(Y,\tilde{Y}) ^{\alpha_2} \\
	& \leq Cr^{\alpha_2}
\end{align*}
for all $X, Y \in Q_1 \cap Q_r(X_0)$ and $k=1,\cdots,n$ and hence we conclude that
\begin{equation*}
	\osc_{Q_1 \cap Q_r(X_0)} D_k u \leq C r^{\alpha_2}
\end{equation*}
for all $r \in (0,1/2)$ and $k=1,\cdots,n$. The case $X_0 \in \partial_b Q_1 \cup \partial_c Q_1$ can be proved in the same way; see \cite{Lie96} for details.
\end{proof}
Finally, by combining \Cref{interiorholder} and \Cref{lem:bd_c1a_n}, we have the following global results.
\begin{lemma}[Global a priori H\"older estimate for the gradient] \label{lem:gbd_c1a}
Suppose that  $\gamma >-2$, and $F$ satisfies \ref{F1} and \ref{F3}. Let $\beta \in (1/2, 1)$ and let $u \in C^{2, \beta}(\overline{Q_1})$ be a solution \eqref{eq:mean} with $\varphi \in C^2(\overline{Q_1})$. Then there exists constant $\alpha \in (0,1)$ depending only on $n$, $\lambda$, $\Lambda$, $\gamma$, $\|u\|_{L^{\infty}(Q_1)}$, and $\|\varphi\|_{C^2(\overline{Q_1})}$ such that 
\begin{align*}
	[Du]_{C^\alpha(\overline{Q_1})} \leq C,
\end{align*}
where $C$ is a constant depending only on $n$, $\lambda$, $\Lambda$, $\gamma$, $\|u\|_{L^{\infty}(Q_1)}$, and $\|\varphi\|_{C^{2}(\overline{Q_1})}$.
\end{lemma}
%
%
\section{$C^{1, \alpha}$-regularity via approximations}\label{sec:approx}
In this section, we consider an approximated solution $u^{\varepsilon}$, $\varepsilon>0$, of a regularized equation which will be specified soon. We first regularize the operator $F$ by standard mollification technique. To be precise, we extend the domain of $F$ from $\mathcal{S}^n$ to $\mathbb{R}^{n^2}$ by considering $F(M)=F\left(\frac{M+M^T}{2}\right)$. We also let $\psi \in C_c^{\infty}(\mathbb{R}^{n^2})$ be a standard mollifier satisfying $\int_{\mathbb{R}^{n^2}}\psi \,dM=1$ and $\mathrm{supp} \,\psi  \subset \{M \in \mathbb{R}^{n^2} : \sum_{i,j=1}^n M_{ij}^2 \leq 1\}$, and define $\psi_{\varepsilon}(M)=\varepsilon^{-n^2}\psi(M/\varepsilon)$. If we define $F^{\varepsilon}$ as
\begin{align*}
	F^{\varepsilon}(M) \coloneqq F \ast \psi_{\varepsilon}(M)=\int_{\mathbb{R}^{n^2}} F(M-N) \psi_{\varepsilon}(N)\,dN.
\end{align*}
It is easy to check that $F^{\varepsilon}$ is uniformly elliptic (with the same ellipticity constants $\lambda, \Lambda$) and smooth. Moreover, $F^{\varepsilon}$ is convex whenever $F$ is convex. Finally, since $F$ is Lipschitz continuous, $F^{\varepsilon}$ converges to $F$ uniformly.

We now consider the following regularized problem:
\begin{equation}\label{eq:reg}
	\partial_tu^{\varepsilon}= (\varepsilon^2+|Du^{\varepsilon}|^2 )^{\gamma/2}F^{\varepsilon}(D^2u^{\varepsilon}).
\end{equation}

In the remaining of this section, we suppose that $F$ satisfies the hypotheses \ref{F1}, \ref{F2}, and \ref{F3}. We first derive uniform $C^{1, \alpha}$-estimates of $u^{\varepsilon}$ for some universal constant $\alpha \in (0,1)$, and then develop $C^{1, \alpha}$-estimate of a viscosity solution $u$ of \eqref{eq:model}, which finishes the proof of \Cref{thm:C1alpha}. For simplicity, we may write $u$ and $F$ instead of $u^{\varepsilon}$ and $F^{\varepsilon}$, if there is no confusion.

\subsection{Uniform Lipschitz estimates}
We begin with uniform log-Lipschitz estimates, whose proof is based on the Ishii-Lions' method \cite{IL90}. It is noteworthy that for \Cref{loglipschitz}, \Cref{lipschitz}, and \Cref{half_hol_t}, both \ref{F2} and \ref{F3} on $F$ are not necessary.
\begin{lemma}[Uniform log-Lipschitz estimate]\label{loglipschitz}
	Let $u$ be a viscosity solution of \eqref{eq:reg} in $Q_1$ with $\varepsilon \in (0,1)$ and  $\gamma >-2$. Then there exist two positive constants $L_1$ and $L_2$ depending only on $n$, $\lambda$, $\Lambda$, $\gamma$, and $\|u\|_{L^{\infty}(Q_1)}$ such that for every $(x_0, t_0) \in Q_{3/4}$, we have
	\begin{align*}
		u(x,t)-u(y,t) \leq L_1|x-y| |\log|x-y||+\frac{L_2}{2}|x-x_0|^2+\frac{L_2}{2}|y-x_0|^2+\frac{L_2}{2}(t-t_0)^2
	\end{align*}
	for all $t \in [t_0-1, t_0]$ and $x, y \in B_{3/4}(x_0)$.
\end{lemma} 

\begin{proof}
	Without loss of generality, we may assume $(x_0, t_0)=(0,0)$. We claim that
	\begin{align}\label{claim}
		m\coloneqq\max_{\substack{x, y \in \overline{B_{3/4}}  \\ t \in [-1, 0] }} \left[u(x, t)-u(y,t)-L_1\phi(|x-y|)-\frac{L_2}{2}|x|^2-\frac{L_2}{2}|y|^2-\frac{L_2}{2}t^2 \right] \leq 0,
	\end{align}
where
\begin{align*}
	\phi(r)\coloneqq
	\begin{cases}
		-r \log r \quad &\text{for $r \in [0, e^{-1}]$}\\
		e^{-1} \quad &\text{for $r \geq e^{-1}$}.
	\end{cases}
\end{align*}
We prove \eqref{claim}  by contradiction: suppose that the positive maximum $m$ is attained at $t \in [-1, 0]$ and $x, y \in \overline{B_{3/4}}$. It immediately follows that $x \neq y$ and
\begin{align*}
	L_1\phi(|x-y|)+\frac{L_2}{2}|x|^2+\frac{L_2}{2}|y|^2+\frac{L_2}{2}t^2 \leq 2\|u\|_{L^{\infty}(Q_1)}.
\end{align*}
In particular, $|x|^2+|y|^2+|t|^2 \leq 12 \|u\|_{L^{\infty}(Q_1)}/L_2$ and
\begin{align*}
	\phi(\theta)\leq \frac{2\|u\|_{L^{\infty}(Q_1)}}{L_1}, \quad \text{where $\theta=|a|$ and $a=x-y$}.
\end{align*}
We choose $L_2=L_2(\|u\|_{L^{\infty}(Q_1)})>0$ large enough to ensure $t \in (-1, 0]$ and $x, y \in B_{3/4}$. Moreover, by choosing $L_1$ sufficiently large so that $\theta  \in (0,e^{-1})$ is small enough, we obtain 
\begin{align}\label{theta}
	\phi(\theta) \geq 2\theta, \quad \phi'(\theta) \geq 1, \quad \text{and so} \quad \theta \leq \frac{\|u\|_{L^{\infty}(Q_1)}}{L_1}.
\end{align}
Furthermore, by applying the parabolic version of Jensen-Ishii's lemma \cite[Theorem 8.3]{CIL92}, we observe that, for every $\varepsilon>0$ sufficiently small, there exist $M_x, M_y \in \mathcal{S}^n$ such that
\begin{enumerate}[label=(\roman*)]
	\item $(a_x, p_x, M_x) \in \overline{\mathcal{P}}^{2, +}u(x, t) \quad \text{and} \quad (-a_y, p_y,  -M_y) \in \overline{\mathcal{P}}^{2, -}u(y, t)$;
	\item 
	$
		\begin{pmatrix}
			M_x & 0 \\
			0 & M_y
		\end{pmatrix}
		\leq L_1
		\begin{pmatrix}
			N & -N \\
			-N & N
		\end{pmatrix}
		+(2L_2+\varepsilon)
		\begin{pmatrix}
			I & 0 \\
			0 & I
		\end{pmatrix}
		;$
	\item $a_x+a_y=L_2t$,
\end{enumerate}
where
\begin{align*}
	p&\coloneqq L_1 \phi'(\theta)\hat{a}, \quad p_x \coloneqq p+L_2x, \quad p_y \coloneqq p-L_2y,\\
	N&\coloneqq\phi''(\theta) \hat{a} \otimes \hat{a}+\frac{\phi'(\theta)}{\theta} (I-\hat{a}\otimes \hat{a}) \quad \text{and} \quad \hat{a}\coloneqq\frac{a}{|a|}=\frac{x-y}{|x-y|}.
\end{align*}
It follows from (i), (iii), and \Cref{lem:semijet} that
\begin{align*}
	L_2t &\leq (\varepsilon^2+|p_x|^2)^{\gamma/2}F(M_x)-(\varepsilon^2+|p_y|^2)^{\gamma/2}F(-M_y) \\
	&= (\varepsilon^2+|p_y|^2)^{\gamma/2} (F(M_x)-F(-M_y) )+
	\left((\varepsilon^2+|p_x|^2)^{\gamma/2}-(\varepsilon^2+|p_y|^2)^{\gamma/2} \right) F(M_x)\\
	&\leq (\varepsilon^2+|p_y|^2)^{\gamma/2}\mathcal{M}_{\lambda, \Lambda}^+(M_x+M_y)+
	\left|(\varepsilon^2+|p_x|^2)^{\gamma/2}-(\varepsilon^2+|p_y|^2)^{\gamma/2} \right| \|M_x\|=:T_1+T_2. 
\end{align*}
Before we estimate two terms $T_1$ and $T_2$, we first provide appropriate $L^{\infty}$-bounds for $p$, $p_x$, $p_y$, $M_x$, and $M_y$. By choosing $L_1$ large enough, we may assume $|p| \gg L_2$ which implies that $|p|/2 \leq |p_x|, |p_y| \leq 2|p|.$
Since the matrix inequality (ii) yields that 
\begin{align*}
	M_x, M_y \leq \left(L_1 \frac{\phi'(\theta)}{\theta} +3L_2\right)I,
\end{align*}
we have 
\begin{align*}
	F(M_x) \geq (\varepsilon^2+|p_x|^2)^{-\gamma/2}L_2t + \left(\frac{\varepsilon^2+|p_y|^2}{\varepsilon^2+|p_x|^2}\right)^{\gamma/2} F(-M_y)\geq -C\left(|p|^{-\gamma}+L_1 \frac{\phi'(\theta)}{\theta}+1 \right),
\end{align*}
where $C$ is a constant depending only on $n$, $\lambda$, $\Lambda$, $\gamma$, and $\|u\|_{L^{\infty}(Q_1)}$. Therefore, we conclude that
\begin{align*}
	\|M_x\|, \|M_y\| \leq C\left(|p|^{-\gamma}+L_1 \frac{\phi'(\theta)}{\theta}+1 \right).
\end{align*}
We are now ready to estimate $T_1$ and $T_2$. For $T_2$, an application of mean value theorem gives
\begin{align*}
	T_2 \leq C|p|^{\gamma-1}\|M_x\| |x+y| \leq C\left(|p|^{-1}+ \frac{|p|^{\gamma}}{\theta}+|p|^{\gamma-1} \right).
\end{align*}
For $T_1$, we again employ the previous matrix inequality given in (ii). First, by evaluating a vector of the form $(\xi, \xi)$ for any $\xi \in \mathbb{R}^n$, we have
\begin{align*}
	(M_x+M_y)\xi \cdot \xi \leq 6L_2|\xi|^2,
\end{align*}
which implies that any eigenvalues of $M_x+M_y$ are less than $6L_2$. Next, by considering a special vector $(\hat{a}, -\hat{a})$, we arrive at
\begin{align*}
	(M_x+M_y)\hat{a} \cdot \hat{a} \leq 4L_1 \phi''(\theta)+6L_2.
\end{align*}
In other words, at least one eigenvalue of $M_x+M_y$ is less than $4L_1 \phi''(\theta)+6L_2$. Therefore, by the definition of the Pucci operator, we have
\begin{align*}
	\mathcal{M}_{\lambda, \Lambda}^+(M_x+M_y) \leq \lambda\left(4L_1 \phi''(\theta)+6L_2\right)+6\Lambda(n-1)L_2\leq C\left(L_1 \phi''(\theta)+1 \right).
\end{align*}
Hence, we obtain
\begin{align*}
	T_1 \leq C\left(L_1\phi''(\theta)+1 \right)|p|^{\gamma} .
\end{align*}
Combining two estimates for $T_1$ and $T_2$, it holds that
\begin{align}\label{finalest}
	-L_1\phi''(\theta) \leq C \left( |p|^{-\gamma-1}+\frac{1}{\theta}+|p|^{-1} +|p|^{-\gamma}+1\right).
\end{align}
By recalling that $\gamma>-2$, $\theta \leq C_0/L_1$, $L_1 \leq |p|=L_1|\phi'(\theta)| \leq -L_1\log\theta$, and $\phi''(\theta)=-1/\theta$, we can choose $L_1$ large enough so that
\begin{align*}
	C \left( |p|^{-\gamma-1}+|p|^{-1} +|p|^{-\gamma}+1\right) \leq \frac{1}{2\theta}L_1.
\end{align*}
It leads to the contradiction if we choose $L_1$ further large enough.
\end{proof}

\begin{lemma}[Uniform Lipschitz estimate]\label{lipschitz}
	Let $u$ be a viscosity solution of \eqref{eq:reg} in $Q_1$ with $\varepsilon \in (0,1)$ and $\gamma>-2$. Then there exist two positive constants $L_1$ and $L_2$ depending only on $n$, $\lambda$, $\Lambda$, $\gamma$, and $\|u\|_{L^{\infty}(Q_1)}$ such that for every $(x_0, t_0) \in Q_{3/4}$, we have
	\begin{align*}
		u(x,t)-u(y,t) \leq L_1|x-y|+\frac{L_2}{2}|x-x_0|^2+\frac{L_2}{2}|y-x_0|^2+\frac{L_2}{2}(t-t_0)^2
	\end{align*}
	for all $t \in [t_0-1, t_0]$ and $x, y \in B_{3/4}(x_0)$.
\end{lemma}

\begin{proof}
	The proof is similar to the one in \Cref{loglipschitz}, but the estimate is improved by exploiting the log-Lipschitz regularity of a viscosity solution; see \cite[Lemma 2.3]{IJS19} for details.
\end{proof}

By letting $t=t_0$ and $y=x_0$ in \Cref{lipschitz} and since $(x_0, t_0)$ is arbitrary, we derive the following Lipschitz estimate for a viscosity solution $u$ of \eqref{eq:reg}:
\begin{align*}
	|u(x, t)-u(y, t)| \leq C|x-y|,
\end{align*}
for every $(x, t), (y, t) \in Q_{3/4}$ with $|x-y| <1$. Here $C>0$ is a constant depending only on $n$, $\lambda$, $\Lambda$, $\gamma$, and $\|u\|_{L^{\infty}(Q_1)}$ and $C$ does not depend on $\varepsilon>0$.

We end this subsection with uniform H\"older estimates in the time variable. The proof relies on the interplay between the regularity in time and space, by applying comparsion principle.
\begin{lemma}[Uniform H\"older estimates in $t$]\label{half_hol_t}
	Let $u$ be a viscosity solution of \eqref{eq:reg} in $Q_1$ with $\varepsilon \in (0,1)$ and $\gamma>-1$. Then there exists a constant $C>0$ depending only on $n$, $\lambda$, $\Lambda$, $\gamma$, and $\|u\|_{L^{\infty}(Q_1)}$ such that
	\begin{align*}
		\sup_{\substack{ (x, t), (x, s) \in Q_{3/4} \\ t \neq s}} \frac{|u(x, t)-u(x,s)|}{|t-s|^{1/2}} \leq C.
	\end{align*}
\end{lemma}
\begin{proof}
	We refer to \cite[Lemma 3.1]{IJS19} or \cite[Lemma 3.3]{FZ23} for the proof. Since only small modification is necessary in the construction of a barrier function, we omit the proof.
\end{proof}
\subsection{Uniform $C^{1, \alpha}$-estimates}
Our aim in this subsection is to provide a uniform estimate of $u^{\varepsilon}$ in $C^{1, \alpha}(\overline{Q_{1/2}})$-norm for a universal constant $\alpha \in (0,1)$ which is independent of $\varepsilon>0$. For this purpose, we will show that $Du$ is H\"older continuous at the point $(0, 0)$ by using the dichotomy strategy employed in \cite{IJS19, JS17}. We point out that by the standard scaling argument with \Cref{lipschitz}, we may assume that $|Du| \leq 1$ in $Q_{3/4}$. For convenience, we further suppose that $|Du| \leq 1$ in $Q_1$, as we could repeat the previous procedure in a larger domain.

%

We first verify a version of `density theorem' by investigating the equation satisfied by $Du$. More precisely, we prove that if the the projection of $Du$ on the unit vector $e \in \mathbb{R}^n$ is away from $1$ in a positive portion of $Q_{1}$, then the behavior of $Du \cdot e$ will be restricted  in $Q_{\tau}$ for some $\tau>0$.
\begin{lemma} \label{lem:Du<1-d}
	Let $u$ be a smooth solution of \eqref{eq:reg} such that $\sup_{Q_1} |Du| \leq 1$. For every $1/2<l<1$ and $\mu>0$, there exist $\tau$, $\delta>0$ depending only on $n$, $\lambda$, $\Lambda$, $\gamma$, $\mu$, and $l$ such that for arbitrary unit vector $e \in \mathbb{R}^{n}$, if
	\begin{align*}
		|\{X \in Q_1 : Du(X) \cdot e \leq l \} >\mu|Q_1|,
	\end{align*}
	then
	\begin{align*}
		Du \cdot e < 1-\delta \quad \text{in $Q_{\tau}^{1-\delta}\coloneqq B_{\tau} \times (-(1-\delta)^{-\gamma}\tau^2, 0].$}
	\end{align*}
\end{lemma}

\begin{proof}
	We first define two quantities, namely, $v \coloneqq |Du|^2$ and $w \coloneqq (Du \cdot e -l+\rho|Du|^2)^+$ for $\rho=l/4$. We also let $\Omega_+ \coloneqq \{X \in Q_1 : w>0\}$. Indeed, we essentially follow the argument based on the Bernstein technique as in the proof of \Cref{globalgradient} and \Cref{interiorholder}. However, the situation becomes simpler, because we a priori assume that $u$ is smooth in this case.
	
	By differentiating \eqref{eq:reg} with respect to $x_k$, we have
	\begin{align*}
		\partial_t u_k =(\varepsilon^2+|Du|^2 )^{\gamma/2} a_{ij}D_{ijk} u +\gamma (\varepsilon^2+|Du|^2 )^{\gamma/2-1}F(D^2u) D_m u D_{mk}u,
	\end{align*}
	where $a_{ij} \coloneqq F_{ij}(D^2u)$. Then we obtain
	\begin{align*}
		(Du \cdot e-l)_t& =(\varepsilon^2+|Du|^2 )^{\gamma/2} a_{ij}D_{ij}(Du \cdot e-l) \\
		&\qquad+\gamma (\varepsilon^2+|Du|^2 )^{\gamma/2-1}F(D^2u)D_m u D_m(Du \cdot e-l)
	\end{align*}
	and
	\begin{align*}
		v_t&=(\varepsilon^2+|Du|^2 )^{\gamma/2} a_{ij} D_{ij}v+\gamma (\varepsilon^2+|Du|^2 )^{\gamma/2-1}F(D^2u) D_m u D_m v \\
		&\qquad\qquad\qquad\qquad\qquad\qquad\qquad\quad-2(\varepsilon^2+|Du|^2 )^{\gamma/2} a_{ij} D_{ki}u D_{kj}u.
	\end{align*}
	Thus, in $\Omega_+$, we deduce that
	\begin{align*}
		w_t&=(\varepsilon^2+|Du|^2 )^{\gamma/2} a_{ij}D_{ij}w+\gamma (\varepsilon^2+|Du|^2 )^{\gamma/2-1}F(D^2u)D_mu D_m w\\
		&\qquad\qquad\qquad\qquad\qquad\qquad\qquad\quad-2(\varepsilon^2+|Du|^2 )^{\gamma/2} a_{ij} D_{ki}u D_{kj}u .
	\end{align*}
	Since $l/2< |Du| \leq 1$ in $\Omega_+$ and $\lambda|\xi|^2 \leq a_{ij}\xi_i\xi_j=F_{ij}(D^2u)\xi_i\xi_j \leq \Lambda|\xi|^2$, it follows that
	\begin{align*}
		w_t \leq \mathcal{M}_{\lambda', \Lambda'}^+(D^2w)+C|Dw|^2 \quad \text{in $\Omega_+$},
	\end{align*}
	where $\lambda'$, $\Lambda'$, $C>0$ depend only on $n$, $\lambda$, $\Lambda$, and $\gamma$. Therefore, $w$ is a viscosity solution of the same inequality in $Q_1$. 
	
	We now choose a constant $c_1>0$ which depend only on $n$, $\lambda$, $\Lambda$, and $\gamma$ such that if we let 
	\begin{align*}
		W \coloneqq 1-l+\rho \quad \text{and} \quad \overline{w} \coloneqq \frac{1}{c_1}\left(1-e^{c_1(w-W)}\right),
	\end{align*}
	then $\overline{w}$ satisfies 
	\begin{align*}
		\overline{w}_t \geq \mathcal{M}^-_{\lambda', \Lambda'}(D^2\overline{w}) \quad \text{in $Q_1$}
	\end{align*}
	in the viscosity sense. Therefore, we arrive at the desired conclusion by applying the weak Harnack inequality for the nonnegative supersolution $\overline{w}$; see \cite{FZ23, IJS19} for details.
\end{proof}

In the rest of the paper, we take $\tau$ small enough to be 
\begin{equation} \label{cond:tau}
	\tau < \min\{ 1-\delta, (1-\delta)^{1+\gamma} )\}.
\end{equation}

If the assumption of \Cref{lem:Du<1-d} holds for all directions $e \in \mathbb{R}^d$ with $|e|=1$, then we can control the oscillation of $Du$ in a smaller parabolic cylinder. The following corollary describes the nice behavior of $Du$ when we can apply \Cref{lem:Du<1-d} in an iterative way.
\begin{corollary}  \label{cor:Du<1-d}
	Let $u$ be a smooth solution of \eqref{eq:reg} such that $\sup_{Q_1} |Du| \leq 1$. For every $1/2<l<1$ and $\mu>0$, there exist $\tau$, $\delta>0$ depending only on $n$, $\lambda$, $\Lambda$, $\gamma$, $\mu$, and $l$ such that, for every nonnegative integer $k \leq \log\varepsilon/\log(1-\delta)$, if
	\begin{equation} \label{inq:ith_step} 
		|\{X \in Q_{\tau^i}^{(1-\delta)^i}: Du(X) \cdot e \leq l (1-\delta)^i \}|> \mu |Q_{\tau^i}^{(1-\delta)^i}| 
	\end{equation}
	for all unit vector $e \in \mathbb{R}^n$ and $i=0,\cdots,k$, then for all $i=0,\cdots,k$, we have
	\begin{equation}\label{eq:cor}
		|Du| < (1-\delta)^{i+1} \quad \text{in $Q_{\tau^{i+1}}^{(1-\delta)^{i+1}}$}.
	\end{equation}
\end{corollary}

\begin{proof}
	We prove by induction. For $i=0$, \Cref{lem:Du<1-d} yields that $Du \cdot e <1-\delta$ in $Q_{\tau}^{1-\delta}$ for all unit vector $e \in \mathbb{R}^n$.
	
	Suppose that \eqref{eq:cor} holds for $i=0, ...,k-1$. If we let 
	\begin{align*}
		\widetilde{u}(x, t) \coloneqq \frac{1}{\tau^k(1-\delta)^k}u(\tau^kx, \tau^{2k}(1-\delta)^{-k\gamma}t),
	\end{align*}
	then $\widetilde{u}$ satisfies
	\begin{align*}
		\widetilde{u}_t =(|D\widetilde{u}|^2+\varepsilon^2 (1-\delta)^{-2k})^{\gamma/2} \widetilde{F}(D^2\widetilde{u}) \quad \text{in $Q_1$},
	\end{align*}
	where
	\begin{align*}
		\widetilde{F}(M)\coloneqq\frac{\tau^k}{(1-\delta)^k} F\left(\frac{(1-\delta)^k}{\tau^k}M\right) \ \text{whose ellipticity constants are the same as $F$.}
	\end{align*}
	We note that $\varepsilon \leq (1-\delta)^k$. Moreover, by the induction hypothesis, we observe that $|D\widetilde{u}| \leq 1$ in $Q_1$ and 
	\begin{align*}
		|\{X \in Q_1: D\widetilde{u}(X) \cdot e \leq l  \}|> \mu |Q_1| \quad \text{for all unit vector $e \in \mathbb{R}^n$}. 
	\end{align*}
	Therefore, by applying \Cref{lem:Du<1-d} for $\widetilde{u}$ and scaling back, we conclude that \eqref{eq:cor} holds for $i=k$.
\end{proof}

We are going to show that $u$ can be approximated by a linear function $L$ when $Du$ is close to some vector $e \in \mathbb{R}^n$ in a large portion. The following lemmas are necessary to control the oscillation between $u$ and $L$ uniformly in the time variable.
\begin{lemma} \label{lem:osc1}
	Let $u \in C(\overline{Q}_1)$ be a viscosity solution of \eqref{eq:reg} with $\gamma>-1$ and $0 < \varepsilon <1$. Assume that for a constant $A>0$, we have
\begin{equation*}
	\underset{x \in B_1}{\textnormal{osc}} \, u(x,t) \leq A \quad \text{for all } t \in [-1,0].
\end{equation*}
Then
\begin{equation*}
	\underset{Q_1}{\textnormal{osc}} \, u \leq C(A+A^{1+\gamma}), 
\end{equation*}
where $C>0$ is a constant depending only on $n$, $\Lambda$, and $\gamma$.
\end{lemma}

\begin{proof}
Consider 
\begin{align*}
	\overline{v}(x) = 
		\begin{cases}
			2A|x|^2 - 5nA(1+16A^2 )^{\gamma/2} \Lambda - u(x,-1) & \text{if } \gamma \ge 0 \\
			2A|x|^\beta - \big(\Lambda(n+\beta-2)(2\beta)^{\gamma+1} + 1\big)  A^{1+\gamma} t - u(x,-1)  & \text{if } -1 < \gamma < 0
		\end{cases}
\end{align*}
and let $\overline{v}(\overline{x}) = \inf_{B_1} \overline{v} $ for some $\overline{x} \in \overline{B}_1$, where $\beta = \frac{2+\gamma}{1+\gamma} >2$. Then the function 
\begin{equation*}
	\overline{w}(x,t) = 
		\begin{cases}
			2A|x|^2 + 5nA(1+16A^2 )^{\gamma/2} \Lambda t - \overline{v}(\overline{x}) & \text{if } \gamma \ge 0  \\
			2A|x|^\beta + \big(\Lambda(n+\beta-2)(2\beta)^{\gamma+1} + 1\big)  A^{1+\gamma} t - \overline{v}(\overline{x}) & \text{if } -1<\gamma<0
		\end{cases}
\end{equation*}
satisfies $\overline{w}(\overline{x},-1) = u(\overline{x},-1)$ and 
\begin{equation} \label{ieq:w>u}
	\overline{w}(x,-1) \ge u(x,-1) \quad \text{for all } x \in \overline{B_1}.
\end{equation}
Here we observe that $\overline{x} \in B_1$; otherwise, we find a contradiction from 
\begin{equation*}
	2A = \overline{w}(\overline{x},-1)-  \overline{w}(0,-1) \leq u(\overline{x},-1)-u(0,-1) \leq \underset{B_1}{\textnormal{osc}} \, u(\cdot,-1) \leq A.
\end{equation*}
We now claim that 
$$\overline{w} \ge u \quad \text{in } Q_1.$$
If not, then there exists $X_0=(x_0,t_0) \in \overline{Q_1}$ such that $ \sup_{Q_1}(u-\overline{w})= u(X_0)-\overline{w}(X_0) >0$. \eqref{ieq:w>u} implies that $t_0 \ne -1$. By considering $\overline{w} + \sup{Q_1}(u-\overline{w})$ and $X_0$ instead of $\overline{w}$ and $(\overline{x},-1)$, we can see $x_0 \in B_1$ due to the same argument above. By recalling that $\overline{w} + \max_{Q_1}(u-\overline{w})$ touches $u$ from above at $X_0$, we have if $\gamma \ge 0$,
\begin{align*}
	5nA(1+16A^2 )^{\gamma/2} \Lambda =\overline{w}_t(X_0) &\leq (\varepsilon^2+|D\overline{w}(X_0)|^2 )^{\gamma/2} F\big(D^2\overline{w}(X_0)\big) \\
	&\leq(1+16A^2 )^{\gamma/2} \mathcal{M}^{+} (4A I_n) \\
	&\leq 4nA(1+16A^2 )^{\gamma/2} \Lambda,
\end{align*}
 and  if $-1<\gamma<0$,
\begin{align*}
	\big(\Lambda(n+\beta-2)(2\beta)^{\gamma+1} + 1\big)  A^{1+\gamma}=\overline{w}_t(X_0) &\leq (\varepsilon^2+|D\overline{w}(X_0)|^2 )^{\gamma/2} F\big(D^2\overline{w}(X_0)\big) \\
	&\leq 2A(\varepsilon^2+4\beta^2 A^2 |x|^{2\beta-2} )^{\gamma/2} \mathcal{M}^{+} (D^2|x|^{\beta}) \\
	&\leq \Lambda (n+\beta-2) (2\beta)^{1+\gamma} A^{1+\gamma}.
\end{align*}
Since it is impossible in both cases, this proves the claim. 

We next consider similarly 
\begin{align*}
	\underline{v}(x) = 
		\begin{cases}
			-2A|x|^2 + 5nA(1+16A^2 )^{\gamma/2} \Lambda - u(x,-1) & \text{if } \gamma \ge 0 \\
			-2A|x|^\beta + \big(\Lambda(n+\beta-2)(2\beta)^{\gamma+1} + 1\big)  A^{1+\gamma} t - u(x,-1)  & \text{if } -1 < \gamma < 0
		\end{cases}
\end{align*}
and let $\underline{v}(\underline{x}) = \sup_{B_1} \overline{v} $ for some $\underline{x} \in \overline{B}_1$. We can also see that the function 
\begin{equation*}
	\underline{w}(x,t) = 
		\begin{cases}
			-2A|x|^2 - 5nA(1+16A^2 )^{\gamma/2} \Lambda t - \underline{v}(\underline{x}) & \text{if } \gamma \ge 0  \\
			-2A|x|^\beta - \big(\Lambda(n+\beta-2)(2\beta)^{\gamma+1} + 1\big)  A^{1+\gamma} t - \underline{v}(\underline{x}) & \text{if } -1<\gamma<0
		\end{cases}
\end{equation*}
satisfies $\underline{w}(\underline{x},-1) = u(\underline{x},-1)$, $\underline{w} \leq u$ in $Q_1$, and 
\begin{align*}
	\underline{w}(x,-1) \leq u(x,-1) \quad \text{for all } x \in \overline{B_1}.
\end{align*}
Since
\begin{align*}
	&\underline{v}(\underline{x}) - \overline{v}(\overline{x}) \\
	&\leq u(\overline{x},-1) - u(\underline{x},-1) + \begin{cases}
		10nA(1+16A^2 )^{\gamma/2} \Lambda & \text{if }\gamma\ge 0 \\
		2\big(\Lambda(n+\beta-2)(2\beta)^{\gamma+1} + 1\big)  A^{1+\gamma}  & \text{if } -1<\gamma<0
	\end{cases} \\
	&\leq \underset{x \in B_1}{\textnormal{osc}} \, u(x,-1) + \begin{cases} 
		10nA(1+16A^2 )^{\gamma/2} \Lambda & \text{if }\gamma\ge 0 \\
		2\big(\Lambda(n+\beta-2)(2\beta)^{\gamma+1} + 1\big)  A^{1+\gamma}  & \text{if } -1<\gamma<0
	\end{cases}\\
	&\leq \begin{cases}
		 A + 10 \cdot 17^{\gamma/2}  n \Lambda \left(A+A^{1+\gamma}\right) & \text{if } \gamma\ge 0\\
		 A + 2\big(\Lambda(n+\beta-2)(2\beta)^{\gamma+1} + 1\big)  A^{1+\gamma}  & \text{if } -1<\gamma<0
	\end{cases} 
\end{align*}
we have
\begin{align*}
	\underset{Q_1}{\textnormal{osc}} \, u & \leq \sup_{Q_1} \overline{w} - \inf_{Q_1} \underline{w} \leq 4A + \underline{v}(\underline{x}) - \overline{v}(\overline{x}) \leq 
		C(A+A^{1+\gamma}),
\end{align*}
where $C>0$ is a constant depending only on $n$, $\Lambda$, and $\gamma$.
\end{proof}
\begin{lemma} \label{lem_osc2}
	Let $u \in C(\overline{Q}_1)$ be a viscosity solution of \eqref{eq:reg} with $\gamma \in \mathbb{R}$ and $0 < \varepsilon <1$ and let $e$ be a unit vector in $\mathbb{R}^n$. Assume that for a constant $A \in (0,1/8)$, we have
\begin{equation*}
	\underset{x\in B_1}{\textnormal{osc}} \, \big(u(x,t) - e \cdot x \big) \leq A \quad \text{for all } t \in [-1,0].
\end{equation*}
Then
\begin{equation*}
	\underset{Q_1}{\textnormal{osc}} \, (u - e \cdot x) \leq  CA,
\end{equation*}
where $C$ is a constant depending only on $n$, $\Lambda$, and $\gamma$.
\end{lemma}

\begin{proof}
	It follows from a similar argument as \Cref{lem:osc1}; see \cite[Lemma 4.5]{IJS19} for details.
\end{proof}

If $Du(O)$ is nonzero, there exists a direction $e$ that does not satisfy \eqref{inq:ith_step}. Roughly speaking, $Du$ and $e$ are close to each other in a set of positive measure, which implies that $u$ can be approximated by some linear function $L$.
\begin{lemma} \label{lem:small_meas}
	Let $\eta>0$ be a constant and let $u$ be a smooth solution of \eqref{eq:reg} with $\gamma >-1$ and $0<\varepsilon <1$ such that $\sup_{Q_1} |Du| \leq 1$. Assume 
\begin{equation} \label{meas_Du-e}
	|\{ X \in Q_1 : |Du(X) -e| > \varepsilon_0 \} | \leq \varepsilon_1 
\end{equation}
for some unit vector $e \in \mathbb{R}^n$ and two positive constants $\varepsilon_0 $, $\varepsilon_1 $. Then if $\varepsilon_0 $ and $\varepsilon_1 $ are sufficiently small, there exists a constant $a \in \mathbb{R}$ such that 
\begin{equation*}
	|u(x,t)-a - e \cdot x| \leq \eta \quad \text{for all } (x,t) \in Q_{1/2},
\end{equation*}
where both $\varepsilon_0 $ and  $\varepsilon_1 $ depend only on $n$, $\Lambda$, $\gamma$, and $\eta$.

\end{lemma}

\begin{proof}
	It follows from the combination of \Cref{half_hol_t}, \Cref{lem:osc1}, and \Cref{lem_osc2}; see \cite[Lemma 4.6]{IJS19} for details.
\end{proof}

\begin{theorem}[Regularity of small perturbation solutions] \label{small_pert_sol}
	Let $u$ be a viscosity solution of \eqref{eq:reg} in $Q_1$. For each $\tilde{\alpha} \in (0,1)$, there exist constant $\eta>0$ depending only on $n$, $\lambda$, $\Lambda$, $\gamma$, $\tilde{\alpha}$, and $\|F\|_{C^{1,1}(\mathcal{S}^n)}$ such that if $|u-L| \leq \eta$ in $Q_1$ for some linear function $L(x)$ satisfying $1/2 \leq |DL| \leq 2$, then $u \in C^{2,\tilde{\alpha}}(\overline{Q_{1/2}})$ and 
$$\|u-L\|_{C^{2,\tilde{\alpha}}(\overline{Q_{1/2}})} \leq C,$$
where $C>0$ is a constant depending only on $n$, $\lambda$, $\Lambda$, $\gamma$, $\tilde{\alpha}$, and $\|F\|_{C^{1,1}(\mathcal{S}^n)}$.
\end{theorem}

\begin{proof}
Since $L$ is a solution of \eqref{eq:reg} in $Q_1$, the conclusion follows from \cite[Corollary 1.2]{Wan13}. 
\end{proof}

We are now ready to prove the H\"older estimate for $Du$ and the H\"older estimate in time variable, which are independent of $\varepsilon>0$.
\begin{theorem}\label{uniformholder}
	Let $u$ be a smooth solution of \eqref{eq:reg} in $Q_1$ with $\gamma>-1$ and $0<\varepsilon<1$ such that $\sup_{Q_1}|Du| \leq 1$. Then there exist constant $\alpha>0$ depending only on $n$, $\lambda$, $\Lambda$, $\gamma$, and $\|F\|_{C^{1,1}(\mathcal{S}^n)}$ such that 
\begin{equation} \label{hol_Du}
	|Du(x,t)-Du(y,s)| \leq C(|x-y|^{\alpha} +|t-s|^{\frac{\alpha}{2-\alpha \gamma}}) \quad \text{for all } (x,t),(y,s) \in Q_{1/2},
\end{equation}
where $C>0$ is a constant depending only on $n$, $\lambda$, $\Lambda$, $\gamma$, and $\|F\|_{C^{1,1}(\mathcal{S}^n)}$. Moreover, there holds 
\begin{equation} \label{hol_time}
	|u(x,t)-u(x,s)| \leq C |t-s|^{\frac{1+\alpha}{2-\alpha \gamma}}  \quad \text{for all } (x,t),(y,s) \in Q_{1/2}.
\end{equation}
\end{theorem}

\begin{proof}
The proof is similar to that of \cite[Theorem 4.8]{IJS19}. However, for completeness, we provide the details of the proof here.

Without loss of generality, by the standard translation argument, it is sufficient to prove \eqref{hol_Du} at $(y,s)=(0,0)$ and \eqref{hol_time} at $(x,s)=(0,0)$. By \Cref{cor:Du<1-d}, for every $l \in (1/2,1)$ and $\mu>0$, there exists $\tau$, $\delta>0$ depending only on $n$, $\lambda$, $\Lambda$, $\gamma$, $\mu$, and $l$ such that for every nonnegative integer $k \leq \log \varepsilon / \log(1-\delta)$, if 
\begin{equation} \label{est:meas}
	|\{X \in Q_{\tau^i}^{(1-\delta)^i}: Du(X) \cdot e \leq l (1-\delta)^i \}|> \mu |Q_{\tau^i}^{(1-\delta)^i}| 
\end{equation}
for all unit vector $e \in \mathbb{R}^n$ and $i=0,\cdots,k$, then
\begin{equation} \label{est:Du_ith}
	|Du| < (1-\delta)^{i+1} \quad \text{in $Q_{\tau^{i+1}}^{(1-\delta)^{i+1}}$} 
\end{equation}
for all $i=0,\cdots,k$. Let $m\coloneqq \min\{m_1, m_2\}$, where $m_1:=[\log \varepsilon/\log(1-\delta)]$ and $m_2$ is defined by the least nonnegative integer such that \eqref{est:meas} does not hold. Here we denote $[z]$ by the integer part of $z \in \mathbb{R}$. Then, for all $p \in \mathbb{R}^n$ with $|p| \leq (1-\delta)^m$, \eqref{est:Du_ith} indicates that
\begin{equation} \label{est:Du-p1}
	|Du(x,t)-p| \leq \frac{2}{1-\delta}(|x|^{\alpha} + |t|^{\frac{\alpha}{2-\alpha \gamma}}) \quad \text{for all } (x,t) \in Q_{\tau^m}^{(1-\delta)^m}\setminus Q_{\tau^{m+1}}^{(1-\delta)^{m+1}},
\end{equation}
where $\alpha= \log(1-\delta)/\log \tau$. If we let
\begin{equation*} 
	\widetilde{u}(x,t) = \frac{1}{\tau^m(1-\delta)^m} u (\tau^m x , \tau^{2m} (1-\delta)^{-m \gamma} t),
\end{equation*}
then $\widetilde{u}$ solves 
\begin{equation*}
	\widetilde{u}_t  =(|D\widetilde{u}|^2+\varepsilon^2 (1-\delta)^{-2m})^{\gamma/2} \widetilde{F}(D^2 \widetilde{u}) \quad \text{in $Q_1$},
\end{equation*}
where
\begin{align*}
	\widetilde{F}(M)\coloneqq\frac{\tau^m}{(1-\delta)^m} F\left(\frac{(1-\delta)^m}{\tau^m}M\right) \ \text{whose ellipticity constants are the same as $F$.}
\end{align*}
Furthermore, since $\|D\widetilde{u}\|_{L^{\infty}(Q_1)}\leq 1$, we observe that
\begin{equation*}
	\underset{x \in B_1}{\textnormal{osc}} \, \widetilde{u}(x,t) \leq 2 \quad \text{for all } t \in [-1,0].
\end{equation*}
 By \Cref{lem:osc1}, we have $ {\textnormal{osc}}_{Q_1} \, \widetilde{u} \leq C$ and hence 
\begin{equation} \label{est:osc1}
	 \underset{Q_{\tau^m}^{(1-\delta)^m}}{\textnormal{osc}} \, u \leq C\tau^m(1-\delta)^m.
\end{equation}
\textbf{(Case 1: $\bm{m=m_1}$.)} In this case, we have $1/4 \leq \varepsilon^2(1-\delta)^{-2m} \leq 1$. Thus, we can apply \Cref{thm:solv}, \Cref{interiorholder}, and the Schauder estimate (see \cite{Wan92b}) to guarantee the existence of $q \in \mathbb{R}^n$ with $|q|\leq 1$ such that
\begin{equation*}
	|D\widetilde{u}(x,t) - q| \leq C(|x| +|t|^{1/2}) \quad \text{for all } (x,t) \in Q_{\tau}^{1-\delta}
\end{equation*}
and $|\widetilde{u}_t|	\leq C$ in $Q_{\tau}^{1-\delta}$. By scaling back, we have
\begin{equation}\label{est:Du-p2}
	|Du(x,t) -(1-\delta)^m q| \leq C(|x| +|t|^{1/2})  \leq C(|x|^{\alpha} + |t|^{\frac{\alpha}{2-\alpha \gamma}}) \quad \text{for all } (x,t) \in Q_{\tau^{m+1}}^{(1-\delta)^{m+1}}
\end{equation}
and 
\begin{equation} \label{est:u-u(0)} 
	|u(x,t) -u(x,0)| \leq C\tau^{-m}(1-\delta)^{m(1+\gamma)} |t| \quad \text{for all } (x,t) \in Q_{\tau^{m+1}}^{(1-\delta)^{m+1}},
\end{equation}
where we used that $\frac{\alpha}{2-\alpha \gamma} < \frac{1}{2}$ for all $\gamma >-1$. Combining \eqref{est:Du-p1} and \eqref{est:Du-p2} gives
\begin{equation*} 
	|Du(x,t) -(1-\delta)^m q|  \leq C(|x|^{\alpha} + |t|^{\frac{\alpha}{2-\alpha \gamma}}) 
\end{equation*}
for all $(x,t) \in Q_{\tau^m}^{(1-\delta)^m}$ which is extensible $Q_{1/2}$. Furthermore, it follows from \eqref{est:u-u(0)} that 
\begin{equation} \label{est:osc2}  
	|u(0,t) -u(0,0)| \leq C\tau^i(1-\delta)^i  \quad \text{for all } (x,t) \in Q_{\tau^i}^{(1-\delta)^i}, \, i > m.
\end{equation}
Similarly, \eqref{est:osc1} and \eqref{est:osc2} implies that
\begin{equation*}
	|u(0,t)-u(0,0)| \leq C|t|^{\frac{1+\alpha}{2-\alpha \gamma}}  \quad \text{for all } t \in (-1/4,0].
\end{equation*}
\textbf{(Case 2: $\bm{m=m_2}$.)} Since $m$ was the nonnegative integer such that \eqref{est:meas} does not hold, we observe
\begin{equation} \label{est:meas_cp} 
	|\{X \in Q_{\tau^m}^{(1-\delta)^m}: Du(X) \cdot e \leq l (1-\delta)^m \}| \leq \mu |Q_{\tau^m}^{(1-\delta)^m}| 
\end{equation}
for some unit vector $e \in \mathbb{R}^n$. By recalling the definition of $\widetilde{u}$, \eqref{est:meas_cp} can be written as 
\begin{equation*}
	|\{X \in Q_1: D\widetilde{u}(X) \cdot e \leq l  \}| \leq \mu |Q_1|. 
\end{equation*}
For $l= 1-\varepsilon_0^2/2$, we can see that 
\begin{equation*}
	\{X \in Q_1: |D\widetilde{u}(X) - e| > \varepsilon_0  \} \subset \{X \in Q_1: D\widetilde{u}(X) \cdot e \leq l  \}
\end{equation*}
and hence if we take $\mu = \varepsilon_1 /|Q_1|$, then we have
\begin{equation*}
	|\{X \in Q_1: |D\widetilde{u}(X) - e| > \varepsilon_0  \}| \leq |\{X \in Q_1: D\widetilde{u}(X) \cdot e \leq l  \}| \leq \varepsilon_1.
\end{equation*}
Since $l$ and $\mu$ were arbitrary, we can take $\varepsilon_0$ and $\varepsilon_1$ small enough. Thus, by \Cref{lem:small_meas}, for a positive constant $\eta$, there exists a constant $a \in \mathbb{R}$ such that 
\begin{equation*}
	|\widetilde{u}(x,t)-a - e \cdot x| \leq \eta  \quad \text{for all } (x,t) \in Q_{1/2}.
\end{equation*}
By \Cref{small_pert_sol}, there exists $q \in \mathbb{R}^n$ such that
\begin{equation*}
	|D\widetilde{u}(x,t) - q| \leq C(|x| +|t|^{1/2}) \quad \text{for all } (x,t) \in Q_{\tau}^{1-\delta}
\end{equation*}
and $|\widetilde{u}_t|	\leq C$ in $Q_{\tau}^{1-\delta}$. Finally, the desired conclusion is obtained in the same way as in the first case.
\end{proof}
\subsection{Proof of \Cref{thm:C1alpha}}
We begin with uniform boundary estimates for smooth solutions of \eqref{eq:reg}.
\begin{lemma}[Boundary estimates]\label{bdryest}
	Let $u \in C(\overline{Q_1}) \cap C^{\infty}(Q_1)$ be a solution of \eqref{eq:reg} with $\gamma>-1$. Let $\varphi \coloneqq u|_{\partial_p Q_1}$ and let $\omega$ be a modulus of continuity of $\varphi$. Then there exists another modulus of continuity $\omega^{\ast}$ which depends only on $n$, $\gamma$, $\lambda$, $\Lambda$, $\omega$, and $\|\varphi\|_{L^{\infty}(\partial_pQ_1)}$ such that
	\begin{align*}
		|u(X)-u(Y)| \leq \omega^{\ast}(d(X, Y))
	\end{align*}
	for all $X, Y \in \overline{Q_1}$.
\end{lemma}

\begin{proof}
	Since the argument essentially follows the lines in \cite[Appendix]{IJS19} with small modification in calculation, we omit the proof.
\end{proof}

The following lemma illustrates that a viscosity solution $u$ of \eqref{eq:model} can be approximated by a sequence of regularized solutions $\{u^{\varepsilon}\}$ of \eqref{eq:reg}. 
\begin{theorem}\label{lem:approx}
	Let $\varphi \in C(\overline{Q_1})$. Then the Dirichlet problem \eqref{eq:mean} is uniquely solvable in $C(\overline{Q_1}) \cap C^{\infty}(Q_1)$. In particular, for any $\varepsilon>0$, the regularized Dirichlet problem 
	\begin{equation*}
		\left\{
		\begin{aligned}
			\partial_tu^{\varepsilon}&= (\varepsilon^2+|Du^{\varepsilon}|^2 )^{\gamma/2}F^{\varepsilon}(D^2u^{\varepsilon}) && \text{in $Q_1$}\\
			u^{\varepsilon}&=u && \text{on $\partial_p Q_1$},
		\end{aligned}\right.
	\end{equation*}
	is uniquely solvable in $C(\overline{Q_1}) \cap C^{\infty}(Q_1)$.
\end{theorem}

\begin{proof}
	Fix $\beta \in (1/2, 1)$. By arguing as in \cite{LSU68, Lie96}, choose a sequence of functions $\{\varphi_k\}_{k=1}^{\infty}$ such that $\varphi_k \in C^{2, \beta}(\overline{Q_1})$,  $\|\varphi_k-\varphi\|_{L^{\infty}(Q_1)} \to 0$, and $\varphi_k$ satisfies the compatibility condition at the corner points:
	\begin{align*}
		\partial_t \varphi_k=(1+|D\varphi_k|^2)^{\gamma/2}F(D^2 \varphi_k) \quad \text{on $\partial_c Q_1$}.
	\end{align*}
	By applying \Cref{thm:solv}, there exists a unique solution $u_k \in C^{2, \beta}(\overline{Q_1})$ of the Dirichlet problem \eqref{eq:mean}. Moreover, an application of \Cref{lipschitz} yields that, for $r \in (0,1)$, there exists a constant $C>0$ which depends only on $n$, $\lambda$, $\Lambda$, $\gamma$, $\|\varphi\|_{L^{\infty}(Q_1)}$, and $r$ such that
	\begin{align*}
		\|Du_k\|_{L^{\infty}(Q_{(r+1)/2})} \leq C.
	\end{align*}
	Then, by applying \Cref{interiorholder}, there exist constants $\alpha \in (0,1)$ and $C>0$ which depend only on $n, \lambda, \Lambda, \gamma$, $\|\varphi\|_{L^{\infty}(Q_1)}$, and $r$ such that 
	\begin{align*}
		\|u_k\|_{C^{1, \alpha}(\overline{Q_r})} \leq C.
	\end{align*}
	Thus, by \Cref{bdryest}, Arzela-Ascoli theorem, and the standard diagonal argument, we can extract a subsequence  $\{u_k\}$ which converges to a limit function $u \in C(\overline{Q_1})$ such that 
	\begin{enumerate}[label=(\roman*)]
		\item $u \in C^{1, \alpha_{r}}(\overline{Q_r})$ for $r \in (0,1)$;
		
		\item $u_k \to u$ uniformly in $C^{1, \alpha'_r}(\overline{Q_r})$ for any $\alpha'_r \in (0, \alpha_r)$.
	\end{enumerate}
	 Since $u_k$ satisfies
	\begin{equation*} 
		\left\{ \begin{aligned}
			\partial_t u_k &= (1+|Du_k|^2 )^{\gamma/2}F(D^2u_k) &&\text{in $Q_1$}\\
			u_k&=\varphi_k &&\text{on $\partial_p Q_1$},
		\end{aligned}\right.
	\end{equation*}
	the stability theorem yields that $u$ satisfies
	\begin{equation*} 
		\left\{ \begin{aligned}
			u_t&= (1+|Du|^2 )^{\gamma/2}F(D^2u) &&\text{in $Q_1$}\\
			u&=\varphi &&\text{on $\partial_p Q_1$}.
		\end{aligned}\right.
	\end{equation*}
	Finally, the Schauder theory and the bootstrap argument guarantee the smoothness of $u$, i.e.,
	\begin{align*}
		u \in C^{1, \alpha} \implies Du \in C^{\alpha} \implies u \in C^{2, \alpha} \implies Du \in C^{1, \alpha} \implies u \in C^{3, \alpha} \quad \cdots.
	\end{align*}
	We refer to \cite{CC95, GT01, Lie96, Wan92b} for details.
\end{proof}

%

We are now ready to prove our main theorem \Cref{thm:C1alpha}.
\begin{proof}[Proof of \Cref{thm:C1alpha}]
	Without loss of generality, we may assume that $u \in C(\overline{Q_1})$. By \Cref{lem:approx}, there exists a unique solution $u^{\varepsilon} \in C(\overline{Q_1}) \cap C^{\infty}(Q_1)$ of 
	\begin{equation*}
		\left\{
		\begin{aligned}
			\partial_tu^{\varepsilon}&= (\varepsilon^2+|Du^{\varepsilon}|^2 )^{\gamma/2}F^{\varepsilon}(D^2u^{\varepsilon}) && \text{in $Q_1$}\\
			u^{\varepsilon}&=u && \text{on $\partial_p Q_1$}.
		\end{aligned}\right.
	\end{equation*}
	As in the proof of \Cref{lem:approx}, we can extract a subsequence $\{u^{\varepsilon_k}\}$ which converges to a limit function $\bar{u} \in C(\overline{Q_1})$ satisfying
		\begin{equation*} 
		\left\{ \begin{aligned}
			\bar{u}_t&=|D\bar{u}|^{\gamma}F(D^2\bar{u}) &&\text{in $Q_1$}\\
			\bar{u}&=u &&\text{on $\partial_p Q_1$}.
		\end{aligned}\right.
	\end{equation*}
	We note that since $u^{\varepsilon_k}$ satisfies \eqref{eq:reg} with $\varepsilon_k \to 0$, one should replace the interior H\"older estimate \Cref{interiorholder} by a uniform one \Cref{uniformholder} in the argument. 
	
	Then the comparison principle \Cref{thm:comp2} tells us $\overline{u}=u$. Moreover, again by \Cref{lipschitz} and \Cref{uniformholder}, there exists a constant $\alpha \in (0, 1)$ which depends only on $n$, $\lambda$, $\Lambda$, $\gamma$, $\|u\|_{L^{\infty}(Q_1)}$, and $\|F\|_{C^{1,1}(\mathcal{S}^n)}$ such that
	\begin{align*}
		\|u^{\varepsilon_k}\|_{C^{1, \alpha}(\overline{Q_{1/2}})} \leq C,
	\end{align*}
where $C>0$ is a constant depending only on $n$, $\lambda$, $\Lambda$, $\gamma$, $\|u\|_{L^{\infty}(Q_1)}$, and $\|F\|_{C^{1,1}(\mathcal{S}^n)}$. By letting $k \to \infty$, we finish the proof.
\end{proof}
	

\bibliographystyle{abbrv}
\bibliography{citation}

\end{document}